\documentclass[reqno,10pt,oneside]{amsart}

\usepackage{amsthm,amsmath,amssymb,amscd}
\usepackage{mathrsfs}
\usepackage{lscape}
\usepackage{graphics} 
\usepackage[all,cmtip]{xy}
\usepackage{graphicx}
\usepackage{subfigure}
\usepackage{url}	
\usepackage{hyperref}
\usepackage{color}
\usepackage[usenames,dvipsnames]{xcolor}
\usepackage{verbatim}
\usepackage{fancyhdr}
\usepackage{geometry}
\usepackage{cancel}
\usepackage[normalem]{ulem}
\usepackage{float}
\usepackage{mathtools}
\usepackage[english]{babel}
\mathtoolsset{showonlyrefs,showmanualtags}
\numberwithin{equation}{section}

\newtheorem{teo}{Theorem}[section]
 
\newtheorem{prop}[teo]{Proposition}

\newtheorem*{problem*}{Problem}

\theoremstyle{definition} 

\theoremstyle{remark}
\newtheorem{remark}[teo]{Remark}

\newtheorem*{notation}{\textbf{Notation}}

\newcommand{\st}{\left(}
\newcommand{\dt}{\right)}
\newcommand{\sq}{\left[}
\newcommand{\dq}{\right]}
\newcommand{\sg}{\left\{}
\newcommand{\dg}{\right\}}
\newcommand{\p}{  {\bf p} }
\newcommand{\q}{ {\bf q} }
\newcommand{\hg}{ {\bf h} }
\newcommand{\kg}{  {\bf k} }
\newcommand{\ag}{ {\bf a}}
\newcommand{\bg}{  {\bf b}}
\newcommand{\cg}{{\bf c}}

\newcommand{\aaa}{{\bf a}}
\newcommand{\bbb}{{\bf b}}
\newcommand{\ccc}{{\bf c}}

\newcommand{\GGG}{ \mathbf G}
\newcommand{\etag}{\boldsymbol{\eta}}
\newcommand{\dd}{\partial\overline{\partial}}
\newcommand{\NN}{\mathbb{N}}

\newcommand{\s}{\textbf{S}}

\newcommand{\vol}{\text{Vol}_{\omega}}

\newcommand{\RR}{\mathbb{R}}
\newcommand{\ZZ}{\mathbb{Z}}
\newcommand{\CC}{\mathbb{C}}

\newcommand{\Sp}{\mathbb{S}}

\newcommand{\dombd}{\mathcal{B}\st\kappa, \delta \dt}

\newcommand{\K}{K\"{a}hler}

\newcommand{\aaaa}{\boldsymbol{\alpha}}
\newcommand{\bbbb}{\boldsymbol{\beta}}
\newcommand{\cccc}{\boldsymbol{\gamma}}

\newcommand{\hko}{{\bf H}_{\hg,\kg}^{out}}

\newcommand{\Lg}{\mathbb{L}_{\omega}}

\newcommand{\Cqdd}{C_{\delta}^{4,\alpha}\st M_{\p} \dt \oplus \mathcal{D}_{\p}\st \bg,\cg \dt}
\newcommand{\Cqddd}{C_{\delta}^{4,\alpha}\st M_{\p} \dt \oplus \mathcal{E}_{\p} \oplus \mathcal{D}_{\p}\st \bg,\cg \dt}

\newcommand{\Cqddq}{C_{\delta}^{4,\alpha}\st M_{\p,\q} \dt \oplus \mathcal{D}_{\p}\st \bg,\cg \dt\oplus\mathcal{D}_{\q}\st \ag \dt}
\newcommand{\Cqdddq}{C_{\delta}^{4,\alpha}\st M_{\p,\q} \dt \oplus \mathcal{E}_{\p} \oplus \mathcal{D}_{\p}\st \bg,\cg \dt\oplus\mathcal{E}_{\q}\oplus \mathcal{D}_{\q}\st \ag \dt}

\newcommand{\Cqdda}{C_{\delta}^{4,\alpha}\st M_{\p,\q} \dt \oplus \mathcal{D}_{\q}\st \bar{\ag} \dt}

\newcommand{\Pbe}{{\bf P}_{\bg,\etag}}

\newcommand{\rep}{r_\varepsilon}

\newcommand{\Rep}{R_\varepsilon}
\newcommand{\hkjj}{H_{\tilde{h},\tilde{k}}^{in}}

\newcommand{\hkii}{{\bf H}_{\tilde{h},\tilde{k}}^{in}}

\newcommand{\Cdx}{C_{\delta}^{4,\alpha}\st X_{\Gamma} \dt }

\newcommand{\etat}{\eta_{\tilde{b},\tilde{h},\tilde{k}}}

\newcommand{\csfii}{f_{\tilde{b},\tilde{h},\tilde{k}}^{in}}
\newcommand{\ega}{e(\Gamma)}
\newcommand{\cga}{c(\Gamma)}
\newcommand{\egaj}{e\st \Gamma_{j}\dt}
\newcommand{\egal}{e\st \Gamma_{N+l}\dt}
\newcommand{\cgaj}{c\st\Gamma_{j}\dt}

\newcommand{\Pbg}{{\bf P}_{\tilde{b},\omega}}

\begin{document}

\title{On the resolution of extremal and constant scalar curvature K\"ahler orbifolds}
\author[Claudio Arezzo] {Claudio Arezzo}
\address{ICTP Trieste and Univ. of Parma, arezzo@ictp.it }
\author{Riccardo Lena}
\address{Univ. of Parma,  riccardo.lena@unipr.it}
\author{Lorenzo Mazzieri}
\address{Scuola Normale Superiore, Pisa, l.mazzieri@sns.it}

\maketitle

\vspace{-,15in}

{\it{1991 Math. Subject Classification:}} 58E11, 32C17.

\section{Introduction}

The aim of this paper is to give a complete answer  to the following question.

\begin{problem*}\label{problema}
Let $\st M,g,\omega \dt$ be a compact orbifold of complex dimension $m$ endowed with an extremal metric $g$, having isolated quotient singularities at points $\{x_j\}_{j=1,..., S}$ and orbifold groups $\Gamma_{j}\triangleleft U(m)$.
Assume in addition that the local models for the  singularities $\CC ^m / \Gamma_j$ admit scalar flat ALE K\"ahler resolutions $\big( X_{\Gamma_{j}} ,h_{j},\eta_{j} \big)$, then we address the following questions:
\begin{itemize}
\item[(i)] Is there an extremal global resolution $\big( \tilde{M}, \tilde{g},\tilde{\omega} \big)$ of the above extremal K\"ahler orbifold, in which suitable neighborhood of the singular points $\{x_j\}_{j=1,..., S}$ are replaced by suitable pieces of the model spaces $\{X_{\Gamma_{j}}\}_{j=1,..., S}$ for the local resolutions?
\smallskip
\item[(ii)]If $g$ is a \K\ constant scalar curvature (Kcsc from now on) metric, when is the solution of the above problem Kcsc as well?
\end{itemize}
\end{problem*}

The above problem, as well as its analogous when one performs blow ups of smooth points, has been the focus of extensive research and it is now well known that the main difficulties arise in the presence of nontrivial holomorphic vector fields, so when 
$H^{0}\st M, TM \dt\neq\sg 0\dg$. 
 It is now well understood that in the Kcsc case their presence forces the points to be in a {\em special symplectic configuration}, in order to get a positive answer to (ii). In other words, the points $\{x_j\}_{j=1,..., S}$ must satisfy a {\em balancing condition} with respect to an $L^2$-orthonormal basis of the symplectic potentials $\{\varphi_i\}_{i=1,...,d}$ of the holomorphic vector fields.
On the other hand, from the works of LeBrun-Simanca \cite{ls}, Arezzo-Pacard-Singer \cite{aps} and Sz\'ekelyhidi \cite{Gabor1}, one expects the extremal problem (i) to be unobstructed. Another evidence in favour of this guess for the extremal problem comes from Tipler's solution to the above question for surfaces with cyclic quotient singularities \cite{Tip}. 

In the extremal case, the key difficulty lies in fact at the beginning, since one needs to construct \K\ potentials for the lift of he holomorphic vector fields on the local models. In the known blow up case, this had been observed in \cite[Proposition 7.3]{aps}  , while Tipler observed this fact in his special case. We prove this crucial fact in complete generality in Proposition \ref{potenzialiHamiltoniani}, and this paves the way to checking that the standard 
gluing procedure works also in this case, finally confirming in complete generality the unobstructedness of the extremal problem. This is shown in Section \ref{matching}.

If $g$ is an extremal metric and $X_{s}$ its extremal vector field, we denote with  $G:=Iso_{0}\st M,g \dt\cap Ham\st M,\omega \dt$   the identity component of the group of Hamiltonian isometries and with $\mathfrak{g}$ its Lie algebra. Moreover we denote with $T\subset G$ the maximal torus whose Lie algebra $\mathfrak{t}$ contains the extremal vector field $X_{s}$ and $\tilde{T}$ its lift to the resolution. 
With these notations, our main result in the extremal case reads:

\begin{teo}\label{maintheoremestremalefacile}
Let $\st M,g,\omega\dt$ be a compact extremal orbifold with  $T$-invariant metric $g$ and singular points $\{x_1, \dots, x_S\}$. Then  there is $\bar{\varepsilon}$ such that for every $\varepsilon \in (0,\bar{\varepsilon})$  the resolution
\[
\tilde{M} : = M \sqcup _{{x_{1}, \varepsilon}} X_{\Gamma_1} \sqcup_{{x_{2},\varepsilon}} \dots
\sqcup _{{x_S, \varepsilon}} X_{\Gamma_S}
\]

\noindent has a $\tilde{T}$-invariant extremal K\"ahler metric.
\end{teo}

Once we have refined our notations and distinguished different types of singular points as below in terms of the asymptotics of the local models, we can give a more precise statement concerning the \K\ class represented by the new extremal metric (see Theorem \ref{maintheoremestremale}). At this point we just underline the fact that the resulting \K\ class heavily depends on the geometry of the local model.

The corresponding Kcsc problem (ii) is definitely more challenging and interesting. Passing from extremal to Kcscs one can rely on Calabi's Theorem \cite{Calabi2}  stating that an extremal metric in a \K\ class with vanishing Futaki invariant is indeed Kcsc. Nonetheless computing the Futaki invariant is never an easy task, and in fact, while computed in Stoppa \cite{Stoppa}, Della Vedova-Zuddas \cite{DVZ} and  Sz\'ekelyhidi \cite{Gabor1} in the smooth blow up case, is completely unknown for general resolutions. On top of this, we believe that the direct approach to the Kcsc equation is of 
 its own interest, from the PDE's point of view.

It is clear from \cite{ap1}, \cite{ap2}, \cite{alm}, that the key technical notion coming into the solution to the above problem is the rate of convergence of the model metrics $\eta_{j}$'s towards the euclidean metric. Indeed a scalar flat ALE K\"ahler metric $\eta$ has an expansion at infinity  of the form (see e.g. \cite{ap1}, Lemma 7.2) 

\begin{equation}
\eta=\begin{cases}i\partial\overline{\partial}\st \frac{|x|^2}{2}+e(X_{\Gamma})|x|^{4-2m} - c(X_{\Gamma}) |x|^{2-2m} +\mathcal{O}\st |x|^{-2m} \dt \dt & m\geq 3\\
&\\
i\partial\overline{\partial}\st \frac{|x|^2}{2}+e(X_{\Gamma})\log\st|x|\dt - c(X_{\Gamma}) |x|^{-2} +\mathcal{O}\st |x|^{-4} \dt \dt & m= 2\,.
\end{cases}
\end{equation}

\noindent The analysis required to construct Kcsc metrics with PDE's methods heavily depends on whether $e(X_{\Gamma})$ vanishes or not. 
It is interesting to observe that this analytic condition does not seem to have an easy algebraic interpretation in terms of the group $\Gamma$. In fact, even in complex dimension $2$, 
there are examples of groups which have (at least) two ALE scalar flat resolutions one with vanishing leading coefficient, and one with non vanishing one (see Le Brun \cite{leb91}, Section 6, page 244, and \cite{RollinSingerII}, Example $2$, Section $6.7$). We wish to thank H.-J. Hein, C. Spotti, C. Le Brun and I. Suvaina for many discussions about this point, and for pointing out to us these examples, which contradict our first guess. The only clear and classical fact is that if $\eta$ is Ricci flat then $e(X_{\Gamma}) = 0$ and $c(X_{\Gamma})$ is negative. It is possible that the condition 
$e(X_{\Gamma}) = 0$ characterises Ricci flat metrics among scalar flat one, on {\em{minimal}} resolutions of the singularities.

\begin{notation}
Despite the existence of the above mentioned examples, to keep the formulae below in a readable shape, we will denote by $e(\Gamma)$ the leading coefficient of the ALE scalar flat metric on $X_{\Gamma}$, i.e. $e(\Gamma) = e(X_{\Gamma})$, since there is no possibility of confusion among different resolutions of the same singularity.

\end{notation}

 \noindent In terms of these coefficients, the following table collects what is known about the main problem in the Kcsc case:
 
 \newpage

\begin{table}[!htbp]
\centering\label{tab2}
\begin{tabular}{p{3cm}| p{3.5cm}| p{3.5cm}| p{3.5cm} } 
& {\begin{center}   $e\st {\Gamma_{j}} \dt=0$ for all $j$'s\end{center}}& {\begin{center}$e\st {\Gamma_{j}}\dt=0$ for some  $j$'s and $e\st {\Gamma_{k}} \dt\neq 0$ for some $k$ \end{center}}   &  {	\begin{center} $e\st {\Gamma_{k}}\dt\neq 0$ for all  $k$\end{center}}\\
\hline\\
{\begin{center} $H^{0}\st M, TM \dt=\sg 0\dg$\end{center}}&{ \begin{center}Yes (\cite{ap1})  \end{center} }&{ \begin{center}Yes  (\cite{ap1})  \end{center} }  &{\begin{center} Yes (\cite{ap1})   \end{center}}\\
&&&\\
\hline\\
{\begin{center} $H^{0}\st M, TM \dt\neq\sg 0\dg$\end{center}}&{ \begin{center}$?_{1}$  \end{center} }&{ \begin{center}$?_{2}$ If $\dim M >2$. \end{center}
\begin{center}Yes  if $\dim\st M\dt=2$ and with points in special symplectic position (\cite{RollinSingerII}) \end{center} }&{\begin{center} Yes with points in special symplectic position  (easy adaptation of \cite{ap2}) \end{center}}\\
\hline
 \end{tabular}
\end{table}

This paper gives the exact balancing conditions to answer $?_{1}$ and $?_{2}$ hence completing the solution to the above problem in terms of the asymptotics of the local models.

\begin{teo}
\label{maintheorem}
Let $(M,g,\omega)$ be a Kcsc orbifold with isolated singularities. Let $\p\:=\sg p_1,\ldots,p_{N}\dg\subseteq M$ the set of points  with neighborhoods biholomorphic to  a ball of  $\CC^m/\Gamma_{j}$ with $\Gamma_{j}$ nontrivial finite subgroup of $U(m)$ such that $\CC^{m}/\Gamma_{j}$ admits a scalar flat ALE resolution $\st X_{\Gamma_{j}},h_j,\eta_{j} \dt$ with $e\st X_{\Gamma_{j}} \dt= 0$. 

\smallskip

\begin{itemize}
\item  Assume $\q:=\sg q_1,\ldots,q_{K}\dg\subseteq M$ is the set of points with neighborhoods biholomorphic to a ball of  $\CC^m/\Gamma_{N+l}$ such that
$\CC^{m}/\Gamma_{N+l}$ admits a scalar flat ALE resolution $(Y_{\Gamma_{N+l}},k_{l},\theta_{l})$ with $e({\Gamma_{N+l}})\neq 0$. If there exist $\ag:=\st a_{1},\ldots,a_{K} \dt\in \st\RR^{+}\dt^{K}$ such that
\begin{equation}
\label{eq:bala}
\begin{cases}
\sum_{l=1}^{K}\frac{a_{l}e\st \Gamma_{N+l} \dt}{|\st \Gamma_{N+l} \dt|}\varphi_{i}\st q_{l} \dt=0 & i=1,\ldots, d\\&\\
\st \frac{a_{l} e(\Gamma_{N+l})}{|e(\Gamma_{N+l})|}  \varphi_{i}\st q_{l} \dt \dt_{\substack{1\leq i\leq d\\1\leq l\leq K}}& \textrm{has full rank}
	\end{cases}
	\end{equation}
then there exists $\varepsilon_0>0$ such that, for any $\varepsilon < \varepsilon_0$ and any ${\bf{b}} = (b_1, \dots, b_n) \in \st\RR^{+}\dt^{N}$, the manifold
\[
{\tilde{M}} : = M \sqcup _{{p_{1}, \varepsilon}} X_{\Gamma_1} \sqcup_{{p_{2},\varepsilon}} \dots
\sqcup _{{p_{N}, \varepsilon}} X_{\Gamma_{N}}\sqcup_{{q_{1}, \varepsilon}} X_{\Gamma_{N+1}} \sqcup_{{q_{2},\varepsilon}} \dots \sqcup _{{q_{N+K}, \varepsilon}}X_{\Gamma_{N+K}},
\]
admits a Kcsc metric.
%
\smallskip

\item If $\q=\emptyset $ and there exists $ \bg \in \mathbb{R}_{+}^{N}$ and $\cg\in\RR^{N}$ such that
			
		\begin{equation}\label{eq:balbc}
		\begin{cases}
		\sum_{j=1}^{N}\st b_{j}\frac{\cgaj}{|\cgaj|}\Delta_{\omega}\varphi_{i}\st p_{j} \dt+c_{j}\varphi_{i}\st p_{j} \dt\dt=0 & i=1,\ldots, d\\
		&\\
		\st \frac{\cgaj}{|\cgaj|}  \st b_{j}\Delta_{\omega}\varphi_{i}\st p_{j} \dt+c_{j}\varphi_{i}\st p_{j} \dt\dt \dt_{\substack{1\leq i\leq d\\1\leq j\leq N}}& \textrm{has full rank}
		\end{cases}
		\end{equation}
with
\begin{align}\label{eq:tuning}
 c_{j}  =&   \,\, b_{j}\sq \frac{1}{m}s_{\omega}\st 1+\frac{\st m-1 \dt^{2}}{\st m+1 \dt} \dt -\frac{\mathtt{c}_{4,j}}{ 2  \st m-1 \dt|\Sp^{2m-1}|} \dq
\end{align}
and the constants $\mathtt{c}_{4,j}$'s  defined in formula \eqref{eq:int4}, then  there exists $\varepsilon_0>0$ such that  for any $\varepsilon < \varepsilon_0$ 
	\[
	\tilde{M} : = M \sqcup _{{p_{1}, \varepsilon}} X_{\Gamma_1} \sqcup_{{p_{2},\varepsilon}} \dots
	\sqcup _{{p_N, \varepsilon}} X_{\Gamma_N}
	\]
admits a Kcsc metric. 
%
\end{itemize}
\end{teo}

\begin{remark}
As a byproduct of our analysis, the \K\ class of the Kcsc metrics produced on $\tilde{M}$ can be specified in terms of the \K\ classes of the building blocks. More precisely, using the same notation as in the statement of Theorem~\ref{maintheorem}, we have that:
\begin{itemize}
\item If $\q \neq\emptyset $, then for every $ \varepsilon < \varepsilon_0$ the Kcsc metrics constructed on 
\[
{\tilde{M}} : = M \sqcup _{{p_{1}, \varepsilon}} X_{\Gamma_1} \sqcup_{{p_{2},\varepsilon}} \dots
\sqcup _{{p_{N}, \varepsilon}} X_{\Gamma_{N}}\sqcup_{{q_{1}, \varepsilon}} X_{\Gamma_{N+1}} \sqcup_{{q_{2},\varepsilon}} \dots \sqcup _{{q_{N+K}, \varepsilon}}X_{\Gamma_{N+K}},
\]
belongs to the class 
$$
\pi^*[\omega]\,  + \, \sum_{l=1}^K\varepsilon^{2m-2} \tilde{a}_{l} ^{2m-2}[\tilde{\theta_{l}}] \, + \, \sum_{j=1}^N\varepsilon^{2m}b_{j} [\tilde{\eta_j}]   
$$ 
where, in terms of the standard inclusions, the $[\tilde{\eta}_{j}]$'s and the $[\tilde{\theta}_{l}]$'s are respectively given by 
\begin{eqnarray*}
\mathfrak{i}_{j}^{*}\sq \tilde{\eta}_{j} \dq=[\eta_{j}] & \hbox{with}&  \mathfrak{i}_{j}:X_{\Gamma_{j},\Rep}\hookrightarrow \tilde{M}\\
\mathfrak{i}_{N+l}^{*}  [ \tilde{\theta}_{l} ] = [\theta_{l}] &\hbox{with}&  \mathfrak{i}_{N+l}:Y_{\Gamma_{N+l},\Rep}\hookrightarrow \tilde{M}
\end{eqnarray*}
whereas the coefficient $\tilde{a}_l$'s are related to the ones appearing in the balancing condition~\eqref{eq:bala}. More precisely they satisfy the estimate
\begin{align}
\left|\tilde{a}_{l}^{2} - \frac{|\Gamma_{N+l}|\, a_{l}}{4 \, |\Sp^{3}| \, |e({\Gamma_{N+l}})|}\right| \, \leq& \,\, C \varepsilon^{\gamma}&\qquad &\textrm{for }m= 2\\
\left|\tilde{a}_{l}^{2m-2} - \frac{|\Gamma_{N+l}| \, a_{l}}{8(m-2)(m-1) \, |\Sp^{2m-1}| \, |e({\Gamma_{N+l}})|}\right|\,  \leq & \,\, C \varepsilon^{\gamma}&\qquad& \textrm{for }  m\geq 3
\end{align}
for some $\gamma>0$.

\smallskip

\item If $\q = \emptyset $, then for every $ \varepsilon < \varepsilon_0$ the Kcsc metrics constructed on 
	\[
	\tilde{M} : = M \sqcup _{{p_{1}, \varepsilon}} X_{\Gamma_1} \sqcup_{{p_{2},\varepsilon}} \dots
	\sqcup _{{p_N, \varepsilon}} X_{\Gamma_N}
	\]
belongs to the class 
$$
\pi^{*}\sq\omega \dq+ \sum_{j=1}^{N}\varepsilon^{2m}\tilde{b}_{j}^{2m}\sq \tilde{\eta}_{j} \dq
$$ 
where the $\sq \tilde{\eta}_{j} \dq$'s satisfy $\mathfrak{i}_{j}^{*}\sq \tilde{\eta}_{j} \dq = [\eta_{j}]$, the maps $\mathfrak{i}_{j}:X_{\Gamma_{j},\Rep}\hookrightarrow \tilde{M}$ being the standard inclusions, whereas the coefficients $\tilde{b}_{j}$ are related to the ones appearing in the balancing conditions~\eqref{eq:balbc} and~\eqref{eq:tuning}. More precisely, they satisfy the estimates
$$
\left|\, \tilde{b}_{j}^{2m} - \frac{|\Gamma_{j}|b_{j}}{2\st m-1 \dt}\, \right| \, \leq \, C \varepsilon^{\gamma} \, 
$$
for some $\gamma>0$.

\end{itemize}
\end{remark}

In light of the above results, it would be very interesting to find a direct algebraic proof of the fact that the equations appearing in the system (\ref{eq:bala}) and (\ref{eq:balbc}) are indeed the vanishing of the Futaki invariant of the \K\ classes we construct.

The strategy of proof follows the line  of the one of \cite[Theorem 1.1]{alm} and \cite[Theorem 1.1]{ap2} and we briefly recall it here for the sake of clearness. 
 The first step is the construction of families -- depending on carefully choosen parameters -- of Kcsc metrics on $M_{\rep}$ (that is the starting orbifold $M$ with small neighborhoods of points in $\p$ and $\q$ removed) and on $X_{\Gamma_{j},\Rep}$'s and $Y_{\Gamma_{N+l}}$'s (that are large compact sets of the local model of resolutions of points respectively of $\p$ and $\q$). The existence of such noncomplete Kcsc metrics follows from equations \eqref{eq:bala} in the first case, and \eqref{eq:balbc} in the second one.

Taking advantage of the possibility of   ``moving" the families of Kcsc metrics by changing their structural parameters, it is possible to find, via an iterative procedure, the correct choice of parameters for which the various families match at the boundaries and produce, at once, the desingularization $\tilde{M}$ together with the Kcsc metric $\tilde{\omega}$.  This step crucially needs equation \eqref{eq:tuning},  when $\q=\emptyset $.

As the reader can note, we  assume the existence of a local model of resolution that is ALE K\"ahler. This is necessary if one tries to perform a gluing construction, but it is indeed a hard problem determining whether a quotient $\CC^{m}/\Gamma$ admits a scalar flat  ALE K\"ahler resolution. In complex dimension $2$ the situation is well understood. Indeed, LeBrun in \cite{LeBrun1}  constructed scalar flat ALE K\"ahler  resolutions for $\CC^{2}/\ZZ_{k}$ with $\ZZ_{k}$ acting diagonally with same weights, Calderbank and Singer in \cite{CalderbankSinger} provided   toric (ALE K\"ahler ) scalar flat  resolutions for $\CC^{2}/\ZZ_{k}$ for any action of $\ZZ_{k}$ and Viaclovsky and Lock in \cite{ViaclovskyLock} settled the case with $\Gamma$ non abelian proving that there exist a scalar flat ALE K\"ahler metric on the minimal resolution of {\em any} $\CC^{2}/\Gamma$. In dimension greater or equal than $3$ the class of finite subgroups of $U(m)$ acting freely on the sphere enlarges greatly and very little is known on the existence of scalar flat ALE K\"ahler metrics on resolutions of the quotient singularities. It is possible to generalize to any dimension the result of \cite{LeBrun1} for group $\ZZ_{k}$ with diagonal action with same weights, and for dimension $3$ a combination of the results of Joyce \cite{j}, Goto \cite{Goto}, Van Coevering \cite{VanCoevering} and Conlon-Hein \cite{ConlonHein} ensure the existence of  Ricci-flat ALE K\"ahler metrics on crepant resolutions of $\CC^{3}/\Gamma$ with $\Gamma$ a finite subgroup of SU(3).

\section{Notations and preliminaries}\label{preliminaries}

From now on $\st  M, g,\omega \dt$ will be a K\"ahler orbifold  with isolated singular points, with extremal metric $g$ and extremal vector field $X_{s}$. We denote with  $G:=Iso_{0}\st M,g \dt\cap Ham\st M,\omega \dt$   the identity component of the group of Hamiltonian isometries and with $\mathfrak{g}$ its Lie algebra.   
Moreover, we denote with $T\subset G$ the maximal torus whose Lie algebra $\mathfrak{t}$ contains the extremal vector field $X_{s}$. It is a standard fact (see \cite{BochnerMartin}) that the action of $T$ can be linearized at fixed points, more precisely it is possible to find adapted K\"ahler normal coordinates in some neighborhood $U$ of fixed point $p$ such that
\begin{equation}
\omega=i\dd\st \frac{|z|^{2}}{2}+\psi_{\omega}\st z \dt \dt\qquad\textrm{ with }\qquad \psi_{\omega}\st z \dt= \mathcal{O}\st |z|^{4} \dt
\end{equation}
and $T$ acts on $U$  as a subgroup of $U\st m \dt$.    
Clearly, singular points are fixed points for the action of $G$ and hence every $\gamma\in G$  lifts to a $\tilde{\gamma}\in Aut_{0} \big(\tilde{M} \big)$ so we denote with $\tilde{G}$ and $\tilde{T}$ the lifts of $G$ and of $\tilde{T}$ to $\tilde{M}$ respectively.

\subsection{Singular points and local resolutions.} We consider then two distinguished  subsets  of  singular points $\p$ and $\q$. We denote with $\p\:=\sg p_1,\ldots,p_{N}\dg\subseteq M$ the set of points  with neighborhoods biholomorphic to  a ball of  $\CC^m/\Gamma_{j}$ with $\Gamma_{j}$ nontrivial finite subgroup of $U(m)$ such that $\CC^{m}/\Gamma_{j}$ admits a scalar flat ALE resolution $\st X_{\Gamma_{j}},h_j,\eta_{j} \dt$ with $\egaj= 0$.  Moreover, we denote with $\q:=\sg q_1,\ldots,q_{K}\dg\subseteq M$ be the set of points with neighborhoods biholomorphic to a ball of  $\CC^m/\Gamma_{N+l}$ such that
$\CC^{m}/\Gamma_{N+l}$ admits a scalar flat  ALE resolution $(Y_{\Gamma_{N+l}},k_{l},\theta_{l})$ with $e({\Gamma_{N+l}})\neq 0$.   Whenever we will work {\em extremal} metrics we assume that groups $\Gamma_{N+l}$ are nontrivial. When there is no risk of confusion and no need to make the above distinction, we will just indicate by $\st X_{\Gamma_{j}},h_j,\eta_{j} \dt$ the local model irrespectfully of the asymptotic of the metric.

\subsection{Eigenfunctions and eigenvalues of $\Delta_{\Sp^{2m-1}}$.}

We denote with  $\Sp^{2m-1}$ is the unit sphere of real dimension $2m-1$, endowed with  metric induced by $(\mathbb{C}^m, g_{eucl})$. We also denote  by $\{\phi_k\}_{k\in \NN}$ a complete orthonormal system of  $L^2(\Sp^{2m-1})$, generated by   eigeinfunctions $\phi_{k}$'s of  $\Delta_{\Sp^{2m-1}}$, so that, for every $k \in \NN$,
$$
\Delta_{\Sp^{2m-1}} \phi_k \,  = \,  \lambda_k \phi_k
$$ 
and  we will indicate with $\Phi_j$ the generic element of the $j$-th eigenspace of $\Delta_{\Sp^{2m-1}}$. For future convenience we introduce the following notation, given $f\in L^{2}\st \Sp^{2m-1} \dt$ we denote  with $f^{(k)}$ the  $L^2\st \Sp^{2m-1} \dt$-projection of $f$ on the $\Lambda_{k}$-eigenspace of $\Delta_{\Sp^{2m-1}}$ and 
\begin{equation}
f^{(\dagger)}:=f-f^{(0)}
\end{equation}

\subsection{The extremal equation}

\noindent We denote by $s_\omega$ the scalar curvature of the  metric $g$ and by $\rho_\omega$ its Ricci form.
We denote moreover with $\s_{\omega}$  the scalar curvature operator  
$$
\mathbf{S}_\omega(\cdot) \, : \, C^\infty(M) \longrightarrow C^\infty(M) \, , \qquad \qquad
 f \,\longmapsto \,\mathbf{S}_{\omega}(f) := s_{\omega+ i\dd f} \, ,
$$ 

We denote with $\mu_{\omega}: M\rightarrow \mathfrak{g}^{*}$   the Hamiltonian moment map ( for a detailed account of Hamiltonian moment map see e.g. \cite{SalamonMcDuff} )  for the action of $G$ on $M$ and we say that it is normalized if
\begin{equation}
\int_{M}\left<\mu_{\omega}, X \right>d\mu_{g}=0 \qquad X\in \mathfrak{g}
\end{equation}
Using an invariant scalar product on $\mathfrak{g}$ and the natural identifications we can regard $\mu_{\omega}$ as 
\begin{equation}
\mu_{\omega}: M\rightarrow TM^{*}\,.
\end{equation}

 As exposed in \cite{LeBrun} the {\em extremal} equation for  $\omega$ corresponds to the prescription
\begin{equation}
\overline{\partial}\partial^{\sharp}s_{\omega}=0
\end{equation}
and in terms of the Hamiltonian moment map this is equivalent to
\begin{equation}\label{eq: estremalegenerale}
s_{\omega}=\left<\mu_{\omega},X_{s}\right>+\frac{1}{\vol\st M \dt}\int_{M}s_{\omega}\,d\mu_{\omega}\,.
\end{equation}
with $X_{s}$ a holomorphic vector field on $M$.

In the sequel we will construct families of metrics in a fixed cohomology class, so it is necessary to understand how  equation \eqref{eq: estremalegenerale} changes if we consider another K\"ahler metric cohomologous to $\omega$.  Once we fix a K\"ahler class  $[\omega]$ and we fix a K\"ahler form  $\omega \in[\omega]$ then the {\em extremal} equation in the class $[\omega]$ is the nonlinear PDE  in the  unknowns  $f\in C^{\infty}(M )$ , $c\in \RR$ and $X\in H^{0}\st M, TM \dt$  

\begin{equation}\label{eq: estremaleclasse}
\s_{\omega}\st f \dt= c + \left<\mu_{\omega+i\dd f},  X\right>+\frac{1}{\vol\st M \dt}\int_{M}s_{\omega}\,d\mu_{\omega}\,.
\end{equation}

\begin{remark}
In equation \eqref{eq: estremaleclasse} appears an unknown constant $c$ because the perturbed moment $\mu_{\omega+i\dd f}$ map is, in general, not normalized.  It is hence  needed this further degree of freedom to obtain the correct {\em extremal} equation.
\end{remark} 

When  studying of a PDE,  it is a standard procedure  to consider the map, between suitable functional spaces, induced by the PDE itself. In the case of equation \eqref{eq: estremaleclasse} the induced map 
\begin{equation}
\mathscr{E}:  \mathcal{D}\subseteq  C^{4,\alpha}\st M \dt\times \RR\times H^{0}\st  M, TM \dt\longrightarrow \RR
\end{equation}
is defined as 
\begin{equation}\label{eq:operatoreestremale}
\mathscr{E}\st f,c,X \dt:=\s_{\omega}\st f \dt- c - \left<\mu_{\omega+i\dd f},  X\right>-\frac{1}{\vol\st M \dt}\int_{M}s_{\omega}\,d\mu_{\omega}\,.
\end{equation}  
and it is a matter of fact that  is highly nonlinear in its variables and the extremal metrics correspond to the triples $\st f,c,X \dt$ such that $\mathscr{E}\st f,c, X \dt=0$.

\vspace{10pt}
From now on we will work in the $T$-invariant framework, so  we  indicate with $C^{k,\alpha} (M )^{T}$ the subset of   $T$-invariant functions in $C^{k,\alpha}( M )$ and the definition of the map $\mathscr{E}$ in the $T$-invariant setting is the obvious one i.e.
\begin{equation}
\mathscr{E}:  \mathcal{D}\subseteq  C^{4,\alpha}\st M \dt^{T}\times \RR\times \mathfrak{t}\longrightarrow \RR
\end{equation}
with $\mathscr{E}$ acting in the same way as above.  The following results, contained in \cite{aps}, are  indispensable to manipulate \eqref{eq:operatoreestremale} in order to solve equation \eqref{eq: estremaleclasse} in  portions of the spaces we will consider in the sequel. 

\vspace{10pt}

It is  natural to look for  a ``Taylor expansion" for $\mathcal{E}$ at a zero, in order to study the behavior of $\mathcal{E}$ for small perturbations of an {\em extremal} metric. The first step is to understand  how the moment map changes as the symplectic form moves in a fixed K\"ahler class and this is done in the following proposition. 

\begin{prop} 
 Let $\st M, g,\omega \dt$ be a K\"ahler manifold with a Hamiltonian action of  a Torus  $T\subset G$ and $f\in C^{\infty}(M )^{T}$  such that 
\begin{equation}
\tilde{\omega}:=\omega+i\dd f
\end{equation} 
is the K\"ahler form of a $T$-invariant K\"ahler metric. A Hamiltonian moment map $\mu_{\tilde{\omega}}$ relative to $\tilde{\omega}$ is 
\begin{equation}
\left<\mu_{\tilde{\omega}}, X \right> := \left<\mu_{\omega},  X \right> -\frac{1}{2}JXf 
\end{equation} 
\end{prop}

With the next proposition we exploit the local structure of map $\mathscr{E}$ in a neighborhood of a zero.

\begin{prop}\label{variazioneestremale}
Let $\st M,g,\omega \dt$ be a compact extremal K\"ahler manifold with extremal vector field $X_{s}$, with $T$-invariant metric $g$ and $\mu_{\omega}$ a normalized moment map for the action of $G$. Let $f\in C^{\infty} (M )^{T} $ such that $\omega+i\dd f$ is a K\"ahler metric. If the triple $\st f,X, c\dt\in C^{4,\alpha}\st M \dt^{T}\times \RR\times \mathfrak{t}$ is sufficiently small i.e.
\begin{equation}
\left\|f\right\|_{C^{4,\alpha} (M )^{T}}+|c|+\left\| X \right\|_{C^{4,\alpha} (M )^{T}}< \mathsf{C}
\end{equation}
for $C>0$ sufficiently small,  then 
\begin{equation}\label{eq:variazioneestremale}
\mathscr{E}\st f, c+ \frac{1}{\vol\st M \dt}\int_{M}s_{\omega}\,d\mu_{\omega}, X_{s}+X\dt=-\frac{1}{2}\Lg\sq f\dq-\frac{1}{2}\left<\nabla s_{\omega}, \nabla f\right>-\left<\mu_{\omega},X\right>+c+\frac{1}{2}JXf+\frac{1}{2}\NN_{\omega}\st f \dt\,.
\end{equation}
where $\Lg$ is given by
\begin{equation}
{\mathbb L}_\omega f \,  =  \, \Delta^{2}_\omega f \,   +  \, 4 \, \langle \, \rho_\omega \, | \,   i\dd f \, \rangle  \, , \label{eq:defLg}
\end{equation}
and $\NN_{\omega}$ is the nonlinear remainder.
\end{prop}

\begin{remark}
The immediate consequence of Proposition \ref{variazioneestremale} is that if we want to solve equation \eqref{eq: estremaleclasse} in a  small neighborhood of an {\em extremal} metric then it is sufficient to solve the following equation 
\begin{equation}\label{eq:estremalecompatta1}
\Lg\sq f\dq+\left<\nabla s_{\omega}, \nabla f\right>+2\left<\mu_{\omega},X\right>\,=\,2c+JXf+\NN_{\omega}\st f \dt\,.
\end{equation}
\end{remark}

The purpose of expanding the nonlinear map $\mathscr{E}$ is twofold: on one side it is  necessary to put in evidence its linearization which  encodes lots of informations on the solution,  on the other side it allows us to put equation \eqref{eq: estremaleclasse} in a form that, under some hypotheses, can be transformed in a fixed point problem and solved by means of contraction theorem.     
In order to  translate equation \eqref{eq:estremalecompatta1} to a fixed point problem,  we need to  to produce a right inverse for the linear operator induced by the linear part of equation \eqref{eq:estremalecompatta1}. It is indeed  necessary  to  identify the kernel of the induced linear operator and this is done in the following proposition.  

\begin{prop}
Let $\st M,g,\omega \dt$ be a compact extremal K\"ahler manifold and let $P_{\omega}: C^{\infty}\st M \dt\rightarrow T^{*}M\otimes TM$ be the differential operator defined by
\begin{equation}
P_{\omega}\sq f  \dq:=-L_{J\nabla f}J\,.
\end{equation}
Then 
\begin{equation}
P_{\omega}^{*}P_{\omega}\sq f\dq=\Lg\sq f\dq+\left<\nabla s_{\omega}, \nabla f\right>
\end{equation}

Moreover, 

\begin{equation}
\ker\st P_{\omega}^{*}P_{\omega} \dt/\RR=\sg \left< \mu_{\omega}, X \right> \,|\, X\in \mathfrak{g}\dg\,.
\end{equation}

and, if we work with $T$-equivariant functions, then  
\begin{equation}
\ker\st P_{\omega}^{*}P_{\omega} \dt/\RR=\sg \left< \mu_{\omega}, X \right> \,|\, X\in \mathfrak{h}\dg\,.
\end{equation}
where $\mathfrak{h}$ is the Lie algebra of $C_{G}\st T \dt$, the centralizer of $T$ in $G$. 
\end{prop}

In light of the previous lemma, the extremal equation for $\omega+i\dd f$ can be rewritten as
\begin{equation}\label{eq:estremalecompatta}
P_{\omega}^{*}P_{\omega}\sq f\dq=2c+2\left<\mu_{\omega},X\right>+JXf+\NN_{\omega}\st f \dt\,.
\end{equation}

\begin{remark}\label{dipendenzaparametri}
As one can  see there are three unknowns in equation \eqref{eq:estremalecompatta} that are $f\in C^{4,\alpha}\st M \dt^{T}$, $c\in \RR$ and $X\in \mathfrak{t}$ that are related one to the other. Indeed to have a solution to the above equation we must have that the right hand side is $L^{2}$-orthogonal to the kernel of $P_{\omega}^{*}P_{\omega}$ and hence the following conditions have to be satisfied 
\begin{equation}
c=-\frac{1}{2\vol\st M \dt}\int_{M}\sq JXf+\NN_{\omega}\st f \dt \dq \,d\mu_{\omega}
\end{equation}  
and
\begin{equation}
\int_{M}\left<\mu_{\omega},X\right>\,\left<\mu_{\omega},X_{i}\right>\,d\mu_{\omega}\,=\,-\frac{1}{2}\int_{M}\sq JXf+\NN_{\omega}\st f \dt \dq \,\left<\mu_{\omega},X_{i}\right>\,d\mu_{\omega}\qquad i=1,\ldots, d
\end{equation}
where  $\sg X_{1},\ldots, X_{d}\dg $ is basis of $\mathfrak{t}$.  
\end{remark}

\subsection{The constant scalar curvature equation}

We recall that the following  expansion holds 
\begin{equation}
\label{eq:espsg}
\mathbf{S}_{\omega}(f) \,\, = \,\,  s_{\omega} - \, \frac{1}{2} \, {\mathbb L}_{\omega}f \, + \, \frac{1}{2} \mathbb{N}_{\omega}(f) \, ,
\end{equation}

We recall  now some well known results on the  structure of  \K\ potentials of Kcsc metrics.  At  any point $p\in M$, there exists a holomorphic coordinate chart  centered at $p$ such that $\omega$ can be written as
$$
\omega \,\, = \,\, i \dd \, \bigg(\,\frac{|z|^2}{2} + \psi_\omega \bigg) \, , \qquad \hbox{with} \qquad \psi_\omega \, = \, \mathcal{O}(|z|^4) \, .
$$
If in addition $s_{\omega}$  is constant, then 
\begin{equation}
\label{eq:decpsig}
\psi_\omega (z, \overline{z}) \, = \,  \sum_{k=0}^{+\infty}\Psi_{4+k}(z, \overline{z}) \, ,
\end{equation}
with $\Psi_{4+k}$  a real homogeneous polynomials  of degree $4+k$ and $\Psi_4$ and $\Psi_5$ satisfying the equations
\begin{align}
\label{eq:bilapp4}
\Delta^2 \, \Psi_4 & =  -2s_\omega \, ,\\
\label{eq:bilapp5} 
\Delta^2 \, \Psi_5 & =  0 \, .
\end{align}

 Let now $\Gamma\triangleleft U(m)$ be  finite,  a complete noncompact \K\ manifold $(X_\Gamma, h, \eta)$ of complex dimension $m$ is an ALE K\"ahler resolution of $\CC^{m}/\Gamma$  if there exist  $R>0$ and a map    
$\pi: X_{\Gamma}\rightarrow \CC^{m}/\Gamma$,
such that 
\begin{equation}
\pi: X_{\Gamma}\setminus \pi^{-1}(B_R)  \longrightarrow \st\CC^{m} \setminus B_{R} \dt/\Gamma
\end{equation}
is a biholomorphism and in standard Euclidean coordinates the metric $\pi_*h$ satisfies the expansion 
\begin{equation}
\left|  \frac{\partial^\alpha}{\partial x^\alpha} \left(  \big (\pi_{*}h)_{i \bar{j}}  \, - \, \frac{1}{2} \, \delta_{i\bar{j}}  \right) \right| \,\, = \,\,  \mathcal{O}\st |x|^{-\tau - |\alpha|}\dt\,, 
\end{equation}
for some $\tau>0$ and every multindex $\alpha \in \NN^m$.

If $(X_\Gamma, h, \eta)$ is  scalar flat then for $R>0$ large enough, we have, by a result in \cite{j}, that on $X_\Gamma\setminus \pi^{-1}(B_R)$ the K\"ahler form can be written as
\begin{equation}\label{eq:eg}
\eta \,\, = \,\,\begin{cases} i\partial\overline{\partial} \st \,  \frac{|x|^2}{2} \, + \, e({\Gamma}) \, |x|^{4-2m}  \, -  \, c({\Gamma}) \, |x|^{2 - 2m} \,  + \, \psi_{\eta}\st x \dt \dt   & \hbox{with} \qquad \psi_\eta \, = \,  \mathcal{O}(|x|^{-2m}) \,\textrm{ for }\, m\geq 3\\
&\\
 i\partial\overline{\partial} \st \,  \frac{|x|^2}{2} \, + \, e({\Gamma}) \, \log\st|x|\dt  \, -  \, c({\Gamma}) \, |x|^{-2} \,  + \, \psi_{\eta}\st x \dt \dt   & \hbox{with} \qquad \psi_\eta \, = \,  \mathcal{O}(|x|^{-4})\,\textrm{ for }\, m=2
 \end{cases}
\end{equation}
 for some real constants $e({\Gamma})$ and $c({\Gamma})$ and the radial component $\psi_{\eta}^{(0)}$ in the Fourier decomposition of $\psi_\eta$ is such that
$$
\psi_{\eta}^{(0)}\st |x| \dt=\mathcal{O}\st |x|^{6-4m} \dt \, .
$$ 
If  moreover $\Gamma \triangleleft U(m)$ is nontrivial and $e\st \Gamma \dt=0$
the K\"ahler form can be written as
\begin{equation}
{\eta} \,\, = \,\, i\partial\overline{\partial} \st \, \frac{|x|^2}{2} \, - \, c({\Gamma}) \, |x|^{2 - 2m} \, + \, \psi_{\eta} \st x \dt \dt    \qquad \hbox{with} \qquad \psi_\eta \, = \,  \mathcal{O}(|x|^{-2m}) \, ,\label{eq: poteta}
\end{equation}
for some non zero real constant $c({\Gamma})$ and, by \cite[Proposition 2.6]{alm} the radial component $\psi_{\eta}^{(0)}$ in the Fourier decomposition of $\psi_\eta$ is such that
$$
\psi_{\eta}^{(0)}\st |x| \dt=\mathcal{O}\st |x|^{2-4m} \dt \, .
$$ 

\begin{remark}
If the ALE K\"ahler resolution $X_{\Gamma}$ admits a Ricci-flat (ALE K\"ahler) metric then by \cite[Theorem 8.2.3]{j} the constant $c\st \Gamma \dt$ is  non negative.
\end{remark}

\section{Linear analysis}\label{linearanalysis}

\subsection{Linear analysis on the base}

\noindent  We indicate by $  M_{\p,\q} \, := \, M\setminus \st \p\cup\q \dt\,$ 
 and we agree that, if $\q=\emptyset$, then $M_{ \p }:=M_{\p,\emptyset}$ and whenever this case occurs and an object, that could be a function or a tensor, has indices relative to elements of $\q$ we set these indices to $0$.  

\noindent The only part of linear analysis needed to attack the extremal problem is the following fact which is proved in \cite{aps}.
\begin{prop}\label{linearizzatoestremalebase}
Let $\st M, g,\omega \dt$ be a compact extremal manifold with $T$-invariant metric $g$. Then, for $\delta\in (4-2m,5-2m)$, the operator  
\begin{equation}
\mathcal{P}_{\omega,\delta}\,:\, C_{\delta}^{4,\alpha}\st M_{\p,\q} \dt^{T}\times \mathfrak{t}\times \RR\longrightarrow C_{\delta-4}^{0,\alpha}\st M_{\p,\q} \dt^{T}
\end{equation}
defined as
\begin{equation}
\mathcal{P}_{\omega,\delta}\sq \st f,X,c \dt \dq \,:=\, P_{\omega}^{*}P_{\omega}\sq f \dq+\left< \mu_{\omega}, X \right>+2c 
\end{equation}
is Fredholm  with 
\begin{equation}
\ker\st \mathcal{P}_{\omega,\delta} \dt/\RR=\mathfrak{t}\,.
\end{equation}
\end{prop}

\noindent  From now on the results contained in this subsection are proved in  \cite{alm} and are {\emph{not}} needed to prove Theorem \ref{maintheoremestremale} but are crucial for Theorem \ref{maintheorem}.

\vspace{10pt}

\noindent  Throughout the paper we will assume that $\Lg$  is $(d+1)$-dimensional and we  set
\begin{equation}
\label{nontrivial_ker}
\ker (\Lg) \,\, = \,\, span_{\RR} \, \{\varphi_0, \varphi_1, \ldots , \varphi_d \} \, ,
\end{equation}
where $\varphi_0 \equiv 1$, $d$ is a positive integer and $\varphi_1, \ldots, \varphi_d$ have zero mean and  $||\varphi_i||_{L^2(M)} = 1$, $i=1, \ldots, d$. 

\begin{itemize}
\item
We need now to recall solvability criteria for equations of the form
\begin{equation}
\Lg u = \mu
\end{equation}
with $\mu$ a linear combination of Dirac delta's and their derivatives and elements of $\ker\st\Lg\dt$.  It was proved in \cite[Proposition 3.4]{alm} that,  if the following {\em linear balancing conditions} hold
\begin{eqnarray}
\label{eq:generalbal}
\sum_{l=1}^K a_l\varphi_i(q_l) \, +\, \sum_{j=1}^Nb_j(\Delta\varphi_i)(p_j) \, + \, \sum_{j=1}^Nc_j\varphi_i(p_j)   &  = &  0 \,, \quad\quad\quad\qquad\qquad \hbox{$i = 1, \dots, d$} \, , \\
\label{eq: balancingbc2}
\sum_{l=1}^K a_l \, + \, \sum_{j=1}^N c_j  &= & \nu_{\aaa,\ccc}  \, {\rm Vol}_\omega(M) \, ,
\end{eqnarray}

there exist a distribution $\GGG_{\aaa,\bbb,\ccc} \in \mathscr{D}'(M)$, which satisfies the equation
\begin{eqnarray*}\label{eq:gabc}
\Lg \left[  \mathbf{G}_{\aaa,\bbb,\ccc} \right] \,  + \, \nu_{\aaa,\ccc}  & =&   \sum_{l=1}^K a_l\, \delta_{q_{l}} \, + \, \sum_{j=1}^N  b_j \, \Delta\delta_{p_{j}} \, + \, \sum_{j=1}^N c_j \, \delta_{p_{j}} \, , \qquad \hbox{in \,\,\,$M$}\, .
\end{eqnarray*}
and we will refer to $\GGG_{\aaa,\bbb,\ccc}$ as a {\em multi-poles fundamental solution} of $\Lg$. We introduce  functions $G_{\Delta\Delta}(q,\cdot) \in  \mathcal{C}^{\infty}_{loc}(M_q)$ for $q\in \q$  which have the espansions at $q$
\begin{equation}
G_{\Delta\Delta}(q,z)  \, =\begin{cases}  |z|^{4-2m} + \, \mathcal{O}(|z|^{6-2m}) & m \geq 3\\
&\\
\log(|z|)  +  C_{q}  +  \mathcal{O}(|z|^{2})&m=2
\end{cases}
\end{equation} 
with  $C_{q}\in \RR$  and functions $G_{\Delta}(p,\cdot) \in \mathcal{C}_{loc}^{\infty}(M_p)$ for $p\in \p$ which have the expansions
\begin{equation}
G_{\Delta}(p,z)  \, =\begin{cases}    |z|^{2-2m}  \, + \, |z|^{4-2m} \, ( \, \Phi_2 + \Phi_4 \, )  \, + \,  |z|^{5-2m} & m \geq 3\\
&\\
|z|^{-2} \, + \, \log(|z|)(\Phi_2 + \Phi_4) \, + \, C_{p} \, + \,  |z| \,\sum_{h=0}^{2}\Phi_{2h+1} \,  + \,  \mathcal{O}(|z|^{2})&m=2
\end{cases}
\end{equation} 
with $C_{p}\in \RR$ and suitable smooth $\Gamma$-invariant functions $\Phi_j$'s defined on $\mathbb{S}^{2m-1}$ and belonging to the $j$-th eigenspace of the operator $\Delta_{\mathbb{S}^{2m-1}}$

\item Given a triple of vectors $\boldsymbol\alpha \in \RR^K$ and $\bbbb, \cccc \in \RR^N$, we set, for $m \geq 3$, $l=1,\ldots,K$ and $j=1, \ldots, N$,
\begin{equation}
W^l_{\aaaa}:=\begin{cases}
 -    \frac{\alpha_l|\Gamma_{N+l}|}{ 8(m-2) (m-1) |\Sp^{2m-1}|}  G_{\Delta\Delta}(q_l,\cdot)  & m\geq 3\\
 &\\
 \frac{\alpha_{l}|\Gamma_{N+l}|}{ 4|\Sp^{3}|} \,\, G_{\Delta\Delta}(q_l,\cdot)  &m=2
\end{cases}
\end{equation}

\begin{align}
W^j_{\bbbb,\cccc}   =\begin{cases} \frac{\beta_j|\Gamma_{j}|}{ 2 (m-1) |\Sp^{2m-1}|} \,\, G_{\Delta}(p_j,\cdot)   - \,  \bigg(     \frac{\gamma_j}{4}  \, - \, \frac{s_\omega \, (m^2-m+2) \, \beta_j}{m(m+1)}  \,  \bigg) \,\, \bigg[   \, \frac{|\Gamma_{j}|}{ 2 (m-2) (m-1) |\Sp^{2m-1}|} \,\,  G_{\Delta\Delta}(p_j,\cdot)  \, \bigg]  & m\geq 3\\
&\\
\beta_j \,\, \bigg[   \, \frac{|\Gamma_{j}|}{ |\Sp^{3}|} \,\,  G_{\Delta}(p_j,\cdot) \, \bigg] \, + \, \bigg(     \frac{\gamma_j}{4}  \, - \, \frac{s_\omega  \, \beta_j}{6}  \,   \bigg) \,\, \bigg[   \, \frac{|\Gamma_{j}|}{ |\Sp^{3}|} \,\, G_{\Delta\Delta}(p_j,\cdot)  \, \bigg]&m=2
\end{cases}
\end{align}
and we define hence the {\em deficiency spaces}
\begin{eqnarray*}\label{deficiency1}
\mathcal{D}_{\q}(\aaaa) \,\, = \,\, \mbox{\em span} \,\Big\{ \, W^l_{\aaaa}\,  : \, {l=1,\ldots , K}  \, \Big\}  \quad & \hbox{and} & \quad
\mathcal{D}_{\p}(\bbbb, \cccc)  \,\,= \,\,  \mbox{\em span}\, \Big\{ \,
W^j_{\bbbb,\cccc} \,  :  \, {j=1, \ldots , N} \, \Big\} \, .
\end{eqnarray*}
 endowed with the following norm. If $V = \sum_{l=1}^K V^l \, W_{\aaaa}^l \in \mathcal{D}_\q(\aaaa)$ and $U = \sum_{j=1}^N U^j W_{\bbbb,\cccc}^j \in \mathcal{D}_\p(\bbbb, \cccc)$, we set
$$
\left\| V \right\|_{\mathcal{D}_\q(\aaaa)} \,\, = \,\, \sum_{l=1}^K\,  | V^l | \qquad \hbox{and} \qquad \left\| U \right\|_{\mathcal{D}_\p(\bbbb,\cccc)} \,\, = \,\, \sum_{j=1}^N\,  | U^j |\, .
$$

For  the case $m=2$, we need to introduce  {\em extra deficiency spaces}  defined as
\begin{eqnarray*}\label{deficiency2}
\mathcal{E}_{\q} \,\, = \,\, \mbox{\em span} \,\big\{ \, \chi_{q_l}\,  : \, {l=1,\ldots , K}  \, \big\}  \quad & \hbox{and} & \quad
\mathcal{E}_{\p}  \,\, = \,\, \mbox{\em span}\, \big\{ \,
\chi_{p_j} \,  :  \, {j=1, \ldots , N} \, \big\} \, ,
\end{eqnarray*}
where the functions $\chi_{p_1}, \ldots, \chi_{p_N},\chi_{q_1}, \ldots, \chi_{q_K}$ are smooth cutoff functions supported on small balls centered at the points $p_{1}, \ldots, p_N, q_1, \ldots, q_K$ and identically equal to $1$ in a neighborhood of these points and they are normed in the following way. Given two functions $X = \sum_{j=1}^N X^j \chi_{p_j}\in \mathcal{E}_{\p}$ and $Y= \sum_{l=1}^K Y^l \chi_{q_l} \in \mathcal{E}_{\q}$, we set
$$
\left\| Y \right\|_{\mathcal{E}_\q} \,\, = \,\, \sum_{l=1}^K\, | Y^l | \qquad \hbox{and} \qquad \left\| X \right\|_{\mathcal{E}_\p} \,\, = \,\, \sum_{j=1}^N\,  | X^j |\, .
$$
  Given a triple of vectors $\boldsymbol\alpha \in \RR^K$ and $\bbbb, \cccc \in \RR^N$, it is convenient to introduce the following matrices
\begin{align}
\Xi_{il}(\boldsymbol{\alpha} ) \,\, := & \,\,\,\, \alpha_{l} \,\frac{\egal}{|\egal|} \varphi_{i}(q_{l}) \,,   & \hbox{for} \quad  i=1 \dots, d \quad \hbox{and} \quad  l=1, \dots, K \, ,
\label{eq:nondeggen2}\\
\Theta_{ij}(\boldsymbol{\beta},\boldsymbol{\gamma} ) \,:=&\,\,\,\frac{\cgaj}{|\cgaj|}\st  \beta_{j} \, \Delta\varphi_{i}(p_{j})  \, + \,
\gamma_{j} \, \varphi_{i}(p_{j})\dt \, , &  \,\hbox{for} \quad i=1 \dots, d \quad \hbox{and} \quad  j=1, \dots, N \, .
\label{eq:nondeggen}
\end{align}
where $\egal$'s are defined in formula \ref{eq:eg}. 

\end{itemize}

\subsection{Linear Analysis on the model}

This subsection is crucial both for the extremal and the Kcsc case.

\label{lineareALE} 
 Let $\st X_{\Gamma},h,\eta \dt$ be an $ALE$ K\"ahler resolution and set 
\begin{equation}X_{\Gamma,R_{0}}=\pi^{-1}\st B_{R_{0}} \dt\,.\end{equation}
where $\pi:X_{\Gamma}\rightarrow \CC^{m}/\Gamma$ is the canonical projection.  Let $\delta\in \RR$, $\alpha\in (0,1)$, the weighted H\"older space $C_{\delta}^{k,\alpha}\st X_{\Gamma} \dt$ is the set of functions $f\in C_{loc}^{k,\alpha}(X_{\Gamma})$ such that
\begin{equation}
\left\|f\right\|_{C_{\delta}^{k,\alpha}\st X_{\Gamma} \dt}:=\left\|f\right\|_{C^{k,\alpha}\st X_{\Gamma,R_{0}} \dt}+\sup_{R\geq R_{0}}R^{-\delta}\left\|f\st R\cdot \dt\right\|_{C^{k,\alpha}\st B_{1}\setminus B_{1/2} \dt}<+\infty\,.
\end{equation} \label{eq:weightedHolderALE}

We can now state the result (\cite[Proposition 4.2, 4.5]{alm}) that summarizes the mapping properties of $\mathbb{L}_{\eta}$ between weighted spaces on ALE K\"ahler spaces.

\begin{prop}\label{isomorfismopesati}
Let $(X_{\Gamma},h,\eta)$ a scalar flat $ALE$ K\"ahler manifold. If $m\geq 3$  and $\delta\in (4-2m,0)$, then 
\begin{equation}
\mathbb{L}_{\eta}^{(\delta)} : C_{\delta}^{4,\alpha}\st X_{\Gamma} \dt\longrightarrow C_{\delta-4}^{0,\alpha}\st X_{\Gamma} \dt
\end{equation}
 is invertible. If $m=2$ and $\delta\in (0,1)$, then
\begin{equation}
\mathbb{L}_{\eta}^{(\delta)} : C_{\delta}^{4,\alpha}\st X_{\Gamma} \dt\longrightarrow C_{\delta-4}^{0,\alpha}\st X_{\Gamma} \dt
\end{equation}
 is surjective with one dimensional kernel spanned by the constant function.
If moreover $e\st \Gamma\dt=0$ and  $\Gamma\triangleleft U(m)$ is nontrivial, then for $\delta\in\st  2-2m , 4-2m \dt$, the operator
\begin{equation}
\mathbb{L}_{\eta}^{(\delta)} : C_{\delta}^{4,\alpha}\st X_{\Gamma} \dt\longrightarrow C_{\delta-4}^{0,\alpha}\st X_{\Gamma} \dt
\end{equation}
has one dimensional cokernel that is orthogonal to the space generated by functions $f\in C_{\delta-4}^{0,\alpha}\st X_{\Gamma} \dt$  satisfying
\begin{equation}
\int_{X_{\Gamma}}f\,d\mu_{\eta}=0\,.
\end{equation}

\end{prop}

\noindent The above Proposition holds trivially also when restricting the domain to $T$-invariant functions.
and it has the following fundamental application in producing the \K\ potentials of the lift of holomorphic vector fields.

\begin{prop}\label{potenzialiHamiltoniani}
Let $\st  X_{\Gamma}, h  , \eta \dt$ be a scalar flat ALE resolution of $\CC^{m}/\Gamma$, then
\begin{equation}
C_{U(m)} \st \Gamma \dt \subset Iso_{0}\st X_{\Gamma},h \dt\,.
\end{equation}

\end{prop} 
\begin{proof}
The standard linear action of $U(m)$ on $\CC^{m}$ induces  a natural action of $C_{U(m)}\st \Gamma \dt$ on $\CC^{m}/\Gamma$ and this action naturally extends to a holomorphic action on the whole  $X_{\Gamma}$.  It is a standard fact that every element of $C_{U(m)}\st \Gamma \dt$ is the exponentiation of an element of $\mathfrak{u}_{\Gamma}\st m \dt \subset\mathfrak{u}\st m \dt$ ( the algebra of  skew symmetric matrices commuting with $\Gamma$) and hence any     $\Xi\in \mathfrak{u}_{\Gamma}\st m \dt$ gives rise to  a holomorphic vector field on $\CC^{m}/\Gamma $, that by abuse of notation we call $\Xi$, and  is of the form
\begin{equation}
\Xi:= \Xi_{i}^{j}z^{i}\partial_{j}\,.   
\end{equation} 
To prove the result we will show that every vector field constructed as above admits a Hamiltonian potential and hence preserves the symplectic structure $\eta$ and consequently, by the K\"ahler condition is also Killing.  The vector field $\Xi$ is Hamiltonian for the standard symplectic structure  $\omega_{eucl}$ indeed it has the purely imaginary potential
\begin{equation}
\left<\mu_{eucl},\Xi\right>:=\Xi_{i}^{j}z^{i}\overline{z^{j}}
\end{equation} 
such that 
\begin{equation}
\overline{\partial}\left<\mu_{eucl},\Xi\right>=\Xi\,\lrcorner\, \omega_{eucl}.
\end{equation} 
Denoting with $\pi\,:\,X_{\Gamma}\longrightarrow \CC^{m}/\Gamma$ the canonical surjection we consider $\xi$  the  $(0,1)$-form on $X_{\Gamma}$  defined as  
\begin{equation}
\xi\,:=\, \tilde{\Xi}\,\lrcorner\,\sq\eta - \pi^{*}\omega_{eucl}\dq
\end{equation}
clearly $\xi\in L^{2}\st X_{\Gamma},h \dt$ and moreover $\overline{\partial}\xi=0$. From \cite[Theorem 8.4.1]{j} we deduce that the first $L^{2}$-cohomology group $H_{L^{2}}^{1}\st X_{\Gamma},\CC \dt$ is isomorphic to the first De Rham cohomology group $H^{1}\st X_{\Gamma},\CC \dt$ and by \cite[Theorem 4.1]{Verbitsky}
we have that $H^{1}\st X_{\Gamma},\CC \dt=\sg 0 \dg$. Moreover by \cite[Corollary 12.8]{DemaillyIllusie} we find that 
\begin{equation}
\xi=\overline{\partial}f
\end{equation}
with $f$ complex function in $L^{2}\st X_{\Gamma},h \dt$. Using the mapping properties of the Laplace operator It is possible to obtain more informations on $f$.  We have, indeed,  that $\xi \in C_{1-2m}^{1,\alpha}\st X_{\Gamma}, \st T^{*}X_{\Gamma}\dt^{(1,0)} \dt$ by construction and hence a simple computation yields  $\overline{\partial}^{*}\xi\in C_{-2m}^{0,\alpha}\st X_{\Gamma},\CC \dt $.  Now
 \begin{equation}
 \Delta_{\eta}f=2\overline{\partial}^{*}\xi
 \end{equation}
and using the surjectivity of the Laplace operator 
\begin{equation}
\Delta_{\eta}: C_{\delta}^{2,\alpha}\st X_{\Gamma},\CC \dt\longrightarrow C_{\delta-2}^{0,\alpha}\st X_{\Gamma},\CC \dt\qquad\delta\in (3-2m,2-2m)
\end{equation}
and the fact that $\Delta_{\eta}$ has no bounded kernel we conclude that  $f\in C_{\delta}^{2,\alpha}\st X_{\Gamma},\CC \dt\cap C_{loc}^{\infty}\st X_{\Gamma} \dt$ with $\delta\in (3-2m,2-2m)$. Now we have, by construction, that 
\begin{equation}
\overline{\partial}\sq \pi^{*}\left<\mu_{eucl},\Xi\right>+f \dq=\tilde{\Xi}\,\lrcorner\,\eta
\end{equation} 
and hence 
\begin{equation}
\mathbb{L}_{\eta}\sq \pi^{*}\left<\mu_{eucl},\Xi\right>+f \dq=0\,.
\end{equation}
Since the metric $h$ is scalar flat, $\mathbb{L}_{\eta}$ is a real operator and consequently both the real and the imaginary part of $\pi^{*}\left<\mu_{eucl},\Xi\right>+f $ are in the kernel of $\mathbb{L}_{\eta}$. The real part of  $\pi^{*}\left<\mu_{eucl},\Xi\right>+f $ coincides with the real part of $f$ since $\pi^{*}\left<\mu_{eucl},\Xi\right>$ is purely imaginary, hence $f$  is bounded and by Proposition \ref{isomorfismopesati} it vanishes identically. We have therefore a Hamiltonian potential for the vector field $\tilde{\Xi}$ with respect to the symplectic form $\eta$. Since the flow of $\tilde{\Xi}$ preserves the complex structure and the symplectic structure, it  preserves also the metric  and the proposition is proved.  
 \end{proof}

We point out that a consequence of  Proposition \ref{potenzialiHamiltoniani} is that $\eta$ is invariant for the action of any torus in $C_{U(m)}\st \Gamma \dt$ in particular for the action of the special torus  $\tilde{T}$ we chose at the beginning. Moreover, as explained in the proof,  given a vector field in $X\in \mathfrak{t}$ and denoted with $\tilde{X}$ its lift to $X_{\Gamma}$, we can always find a Hamiltonian potential $\left<\mu_{\eta},\tilde{X}  \right>$ such that

\begin{equation}
\overline{\partial}\left<\mu_{\eta},\tilde{X}  \right>= \tilde{X}\,\lrcorner\, \eta\,.
\end{equation} 
This is a remarkable fact, indeed, a priori there could be obstructions,  topological or analytical, to the existence of Hamiltonian potentials as it happens in the asymptotically conical  setting. It turns out instead  that, because of the special nature of the singularity $\CC^{m}/{\Gamma}$, it is possible to find, on the resolutions, Hamiltonian potentials  of all  holomorphic vector fields coming from linear fields on $\CC^{m}/\Gamma$.

\section{Nonlinear analysis}

In this section we collect all the estimates needed in the proof of Theorem\ref{maintheorem}. We introduce a small parameter $\varepsilon$ and we will work on  the truncated 
manifolds (or orbifolds) $M_{\rep}$  and $X_{\Gamma_{j},\Rep}$ for $j=1,\ldots, N $  where we impose the  relation:
\begin{equation}
\rep=\varepsilon^{\frac{2m-1}{2m+1}}=\varepsilon \Rep.
\end{equation}

\begin{notation}
{\em For the rest of the section $\chi_{j}$ will denote a  smooth cutoff functions identically $1$ on $B_{2r_{0}}\st p_{j} \dt$ and identically $0$ outside $B_{3r_{0}}\st p_{j} \dt$ and $\chi_{l}$ will denote a  smooth cutoff functions identically $1$ on $B_{2r_{0}}\st q_{l} \dt$ and identically $0$ outside $B_{3r_{0}}\st q_{l} \dt$  }
\end{notation}

\subsection{{\em Pseudo-boundary data} and Biharmonic extensions}\label{biharmonicext}

We recall here the definition of {\em Pseudo-boundary data}.

\begin{align}
\mathcal{B}_{j}:=&C^{4,\alpha}\st \Sp^{2m-1}/\Gamma_{j} \dt\times C^{2,\alpha}\st \Sp^{2m-1}/\Gamma_{j} \dt\\
\mathcal{B}_{N+l}:=&C^{4,\alpha}\st \Sp^{2m-1}/\Gamma_{N+l} \dt\times C^{2,\alpha}\st \Sp^{2m-1}/\Gamma_{N+l} \dt
\end{align}

\begin{equation}
\mathcal{B}:=\st\prod_{j=1}^{N}\mathcal{B}_{j}\dt\times \st\prod_{l=1}^{K}\mathcal{B}_{N+l}\dt
\end{equation}

for $\q\neq \emptyset$
\begin{equation}\label{eq:dombdq}
\mathcal{B}\st   \kappa, \delta\dt:=\sg \st \hg,\kg \dt \in \mathcal{B}\, |\,  \left\| h_{j},k_{j}\right\|_{\mathcal{B}_{j}}\, ,\,  \left\|h_{N+l},k_{l}\right\|_{\mathcal{B}_{N+l}}\leq \kappa \rep^{4}  \dg
\end{equation}

and for $\q=\emptyset$

\begin{equation}\label{eq:dombd}
\mathcal{B}\st  \kappa, \delta\dt:=\sg \st \hg,\kg \dt\in \mathcal{B}\, \left\|\, \left\|h^{(0)}_{j},k^{(0)}_{j}\right\|_{\mathcal{B}_{j}}\leq \kappa \varepsilon^{4m+2}\rep^{-6m+4-\delta},  \left\|h^{(\dagger)}_{j},k^{(\dagger)}_{j}\right\|_{\mathcal{B}_{j}}\leq \kappa \varepsilon^{2m+4}\rep^{2-4m-\delta}\right.  \dg  
\end{equation}

Let $(h,k)\in C^{4,\alpha}\st\Sp^{2m-1}\dt\times C^{4,\alpha}\st\Sp^{2m-1}\dt$ the outer biharmonic extension of $(h,k)$ is the function $H_{h,k}^{o}\in C^{4,\alpha}\st \CC^{m}\setminus B_{1} \dt$ solution fo the boundary value problem
\begin{equation}\begin{cases}
\Delta^{2} H_{h,k}^{out}=0 & \textrm{ on } \CC^{m}\setminus B_{1}\\
H_{h,k}^{out}=h &\textrm{ on } \partial B_{1}\\
\Delta H_{h,k}^{out}=k& \textrm{ on } \partial B_{1}
\end{cases}\end{equation}
Moreover $H_{h,k}^{out}$ has the following expansion in Fourier series 
\begin{equation}\label{eq:bihout2}
H^{out}_{h,k}:=\begin{cases}
\sum_{\gamma=0}^{+\infty}\st  \st  h^{(\gamma)}+\frac{k^{(\gamma)}}{4(m+\gamma-2)}  \dt |w|^{2-2m-\gamma}-\frac{k^{(\gamma)}}{4(m+\gamma-2)}|w|^{4-2m-\gamma} \dt \phi_\gamma & m\geq 3\\
&\\
h^{(0)}|w|^{-2}+\frac{k^{(0)}}{2}\log\st|w|\dt+  \sum_{\gamma=1}^{+\infty}\st  \st  h^{(\gamma)}+\frac{k^{(\gamma)}}{4\gamma}  \dt |w|^{-2-\gamma}-\frac{k^{(\gamma)}}{4\gamma}|w|^{-\gamma} \dt \phi_\gamma & m=2
\end{cases}
\end{equation}

 Let $\st \tilde{h}, \tilde{k}\dt\in C^{4,\alpha}\st \Sp^{2m-1} \dt\times C^{2,\alpha}\st \Sp^{2m-1} \dt$, the biharmonic extension $\hkjj$ on $B_1$ of $\st \tilde{h}, \tilde{k}\dt$ is the function $\hkjj\in C^{4,\alpha}\st \overline{B_{1}} \dt$ given by the solution of the boundary value problem
\begin{equation}\begin{cases}
\Delta^{2}H_{\tilde{h},\tilde{k}}^{in}=0 & w\in  B_{1} \\
H_{\tilde{h},\tilde{k}}^{in}=\tilde{h}& w\in \partial B_{1}\\
\Delta H_{\tilde{h},\tilde{k}}^{in}=\tilde{k}& w\in \partial B_{1}\\
\end{cases}\,.\end{equation}

\noindent The function $\hkjj$  has moreover the expansion

\begin{equation}\label{eq:bihout3}
\hkjj\st w\dt=\sum_{\gamma=0}^{+\infty}\left(\left(\tilde{h}^{(\gamma)}-\frac{\tilde{k}^{(\gamma)}}{4(m+\gamma)}  \dt |w|^\gamma+\frac{\tilde{k}^{(\gamma)}}{4(m+\gamma)}|w|^{\gamma+2} \dt \phi_\gamma\,.
\end{equation}

\begin{remark}
In the sequel we will take $\Gamma$-invariant $(h,k)\in C^{4,\alpha}\st\Sp^{2m-1}\dt\times C^{4,\alpha}\st\Sp^{2m-1}\dt$ and by \cite[Remark 2.4]{alm}  we will 
have no terms with $\phi_{1}$ in the formula  \eqref{eq:bihout2} for nontrivial $\Gamma$.
\end{remark}

\subsection{Extremal metrics on the truncated base orbifold.}

We want to construct perturbations of $\omega$ on $M_{\rep}$ of the form
\begin{equation}
\omega_{\hg,\kg}:=\omega+i\dd F_{\hg,\kg}^{out}
\end{equation}
such that $\omega_{\hg,\kg}$ is extremal. Since we look for small perturbations, by Proposition \ref{variazioneestremale} we have to find $\st F_{\hg,\kg}^{out},X^{out},c^{out} \dt\in C^{\infty}\st M_{\rep} \dt^{T}\times \mathfrak{t}\times \RR $ that solves equation \eqref{eq:estremalecompatta}  i.e.   
\begin{equation}\label{eq:estremaletroncatobase}
P_{\omega}^{*}P_{\omega}\sq F_{\hg,\kg}^{out} \dq+2\left<\mu_{\omega},X^{out}\right>+2c^{out} =  -JX^{out}\sq F_{\hg,\kg}^{out} \dq+\NN_{\omega}\st F_{\hg,\kg}^{out} \dt
\end{equation}

%
%
We seek $F_{\hg,\kg}^{out}$ of the form
\begin{align}
F_{\hg,\kg}^{out}:=&{\bf H}_{\hg,\kg}^{out} +  f_{\hg,\kg}^{out}
\end{align}
with 
\begin{align}
{\bf H}_{\hg,\kg}^{out}:=&\sum_{i=1}^{N+K}\chi_{i}H_{h_{i},k_{i}}^{out}	\st\frac{z}{\rep} \dt 
\end{align}

As usual, it is convenient to work on the punctured orbifold $M_{\p,\q} $ and to use the standard truncation/extension operators $\mathcal{E}_{\rep}$, so  we want to find  $\st f_{\hg,\kg}^{out},X^{out}, c^{out}  \dt \in C_{\delta}^{4,\alpha}\st M_{\p,\q} \dt^{T}\times \mathfrak{t}\times \RR $  with $\delta \in (4-2m,5-2m)$ such that

\begin{equation}\label{eq:estremalebase}
P_{\omega}^{*}P_{\omega}\sq f_{\hg,\kg}^{out} \dq+2\left<\mu_{\omega},X^{out}\right>+2c^{out} =-\mathcal{E}_{\rep}P_{\omega}^{*}P_{\omega}\sq {\bf H}_{\hg,\kg}^{out} \dq  -\mathcal{E}_{\rep}JX^{out}\sq  {\bf H}_{\hg,\kg}^{out} +  f_{\hg,\kg}^{out} \dq+\mathcal{E}_{\rep}\NN_{\omega}\st {\bf H}_{\hg,\kg}^{out} +  f_{\hg,\kg}^{out} \dt
\end{equation}

Making use of Proposition \ref{linearizzatoestremalebase} and Proposition \ref{isomorfismopesati} we can we rephrase equation \eqref{eq:estremalebase}  as a fixed point problem
\begin{align}
\st f_{\hg,\kg}^{out}, X^{out}, c^{out}  \dt \,\,=&\,\, \mathcal{N}^{out}\st f_{\hg,\kg}^{out},X^{out},c^{out},\hg,\kg \dt
\end{align}

with 
\begin{align}
\mathcal{N}^{out}\,:\,& C_{\delta}^{4,\alpha}\st M_{\p,\q} \dt^{T}\times \mathfrak{t}\times \RR\times \mathcal{B}\st \kappa,\delta \dt\rightarrow C_{\delta}^{4,\alpha}\st M_{\p,\q} \dt^{T}\times \mathfrak{t}\times \RR
\end{align}
nonlinear continuous operator. With computations analogous to those performed in \cite{aps} one  obtains the following result.
\begin{prop}\label{famigliabase}
Let $\st M,g,\omega \dt$ be an extremal orbifold with $T$-invariant metric $g$.Then for every $\st \hg,\kg \dt\in \mathcal{B}\st \kappa \dt$ there is $\st  f_{\hg,\kg}^{out},X^{out},c^{out} \dt\in C_{\delta}^{4,\alpha}\st M_{\p,\q} \dt^{T}\times \mathfrak{t}\times \RR$ with 
\begin{equation}\label{eq:boundestremalebase}
\left\|f_{\hg,\kg}^{out}\right\|_{C_{\delta}^{4,\alpha}\st M_{\p,\q} \dt}+\left\| X^{out}\right\|_{C^{4,\alpha}\st M \dt}+|c^{out}|\leq C\st \kappa \dt \rep^{2m}
\end{equation}
such that 
\begin{equation}
\tilde{\omega}_{\hg,\kg}:=\omega+i\dd\st  {\bf H}_{\hg,\kg}^{out} +  f_{\hg,\kg}^{out} \dt
\end{equation}
is an extremal K\"ahler metric on $M_{\rep}$ with extremal vector field $X_{s}+X^{out}$. 
\end{prop}

It is important to note that a consequence of condition \eqref{eq:boundestremalebase} is that  around points in $\p$ and $\q$ the following estimates hold
\begin{equation}
\left\|\left. f_{\hg,\kg}^{out}\st \rep\cdot \dt\right|_{B_{2\rep}\st p_{j} \dt}  \right\|_{C^{4,\alpha}\st \overline{B_{2}}\setminus B_{1} \dt},\left\|\left. f_{\hg,\kg}^{out}\st \rep\cdot \dt\right|_{B_{2\rep}\st q_{N+l} \dt}  \right\|_{C^{4,\alpha}\st \overline{B_{2}}\setminus B_{1} \dt}\leq \mathsf{C}\rep^{4}
\end{equation} 
with $\mathsf{C}$ positive constant depending only on $g$. This kind of estimate is indeed necessary for the success of the data matching procedure. 
\vspace{10pt}

Another fact that that deserves a word of comment is the fact that the triple $\st  f_{\hg,\kg}^{out},X^{out},c^{out} \dt$ of Proposition \ref{famigliabase} depends nonlinearly unpon the parameters $\st \hg,\kg \dt$.  Indeed, being a solution of the PDE \eqref{eq:estremalebase} it is immediate to see, by the structure of such equation,  the nonlinear dependence of $f^{out}$ on $\st \hg,\kg \dt$ and then Remark \ref{dipendenzaparametri} evidence the nonlinear dependence of $c^{out}$ and $X^{out}$ on $\st \hg,\kg \dt$ and the link between $  f_{\hg,\kg}^{out}$, $X^{out}$ and $c^{out}$.

\subsection{Kcsc metrics on the truncated base orbifold.}\label{nonlinbase}

\noindent We recall here the  construction of the families of Kcsc metrics on  $ M_{\rep}$. We want to construct $F_{\ag,\bg,\cg,\hg,\kg}^{out}\in C^{4,\alpha}\st M_{\rep} \dt$ such that 
\begin{equation}\omega_{\ag,\bg,\cg,\hg,\kg}:=\omega +i\dd F_{\ag,\bg,\cg,\hg,\kg}^{out}\end{equation} 

\noindent is a metric on $M_{\rep}$ and its scalar curvature $s_{\omega_{\ag,\bg,\cg,\hg,\kg}}$ is a small perturbation of the scalar curvature $s_{\omega}$ of the reference K\"ahler metric on $M$.

  The function $F_{\ag,\bg,\cg,\hg,\kg}^{out}$ consists of blocks and takes different shapes whether $\q$ is empty or not 
\begin{equation}
F_{\ag,\bg,\cg,\hg,\kg}^{out}:=\begin{cases} \varepsilon^{2m-2}{\bf{G}}_{\ag,{\bf 0},\cg}+ \hko+ f_{\ag,{\bf 0},\cg,\hg,\kg}^{out}& \q\neq \emptyset\\
&\\
-\varepsilon^{2m}{\bf{G}}_{{\bf 0} ,\bg,\cg}+\Pbe+ \hko+ f_{{\bf 0},\bg,\cg,\hg,\kg}^{out}&\q=\emptyset
\end{cases}\,.
\end{equation}  
The function $F_{\ag,\bg,\cg,\hg,\kg}^{out}$  is made of  the  skeleton $\varepsilon^{2m-2}{\bf{G}}_{\ag,\bg,\cg}$ , respectively $\varepsilon^{2m}{\bf{G}}_{\bf{0},\bg,\cg}$ when $\q=\emptyset$, biharmonic extensions of {\em pseudo-boundary data} $\hko$, transplanted potentials of $\eta_{j}$'s $\Pbe$ when $\q=\emptyset$ and a ``small" correction term $f_{\ag,\bg,\cg,\hg,\kg}^{out}$ that has to be determined.  
\begin{itemize}
\item[] {\bf Skeleton.} The skeleton is made of {\em multi-poles fundamental solutions} ${\bf{G}}_{\ag,\bg,\cg}$  of $\Lg$ introduced in formula  \ref{eq:gabc} that near points  $p_{j}$ we have the expansion      
\begin{equation}
{\bf{G}}_{\ag,\bg,\cg}\sim \frac{ \cgaj b_{j} |\Gamma_{j}|}{2 |\cgaj| \st m-1 \dt |\Sp^{2m-1}|}G_{\Delta}\st p_{j},z \dt.
\end{equation}
and near points $q_{l}$
\begin{equation}
{\bf{G}}_{\ag,\bg,\cg}\sim \frac{a_{l}\egal |\Gamma_{l}|}{8\st m-2 \dt\st m-1 \dt |\egal| |\Sp^{2m-1}|}G_{\Delta\Delta}\st q_{l},z \dt.
\end{equation}  
where $ \cgaj  $ and $\egal$ are introduced in formulas \eqref{eq:eg}  and \eqref{eq: poteta}. It is then convenient, from now on, to set the following notation
\begin{align}\label{eq:Al}
A_{l}=&\st \frac{a_{l}|\Gamma_{l}|}{ 8|\egal| \st m-2 \dt \st m-1 \dt|\Sp^{2m-1}|}  \dt^{\frac{1}{2m-2}}\\
\label{eq:Bj}
B_{j}=&\st \frac{b_{j}|\Gamma_{j}|}{ 2|\cgaj | \st m-1 \dt|\Sp^{2m-1}|}  \dt^{\frac{1}{2m}}\\
\label{eq:Cj}
C_{j}=&\frac{|\Gamma_{j}|}{8\st m-2 \dt\st m-1 \dt}\sq  2|\cgaj|  B_{j}^{2m}\frac{\st m-1 \dt|\Sp^{2m-1}|}{m|\Gamma_{j}|}s_{\omega}\st 1+\frac{\st m-1 \dt^{2}}{\st m+1 \dt} \dt-c_{j} \dq.
\end{align}

\item[] {\bf Extensions of {\em pseudo-boundary data}.} As in \cite{alm}, \cite{ap1} and \cite{ap2}   we define  for $\st \hg,\kg \dt\in \dombd$ 
\begin{equation}\label{eq:hko}
\hko:=\begin{cases}\sum_{j=1}^{N}\chi_{j}H_{h_{j}^{(\dagger)},k_{j}^{(\dagger)}}^{out}\st \frac{z}{\rep} \dt & \q\neq\emptyset\\
&\\
\sum_{j=1}^{N}\chi_{j}H_{h_{j},k_{j}^{(\dagger)}}^{out}\st \frac{z}{\rep}\dt+\sum_{l=1}^{K}\chi_{l}H_{h_{N+l},k_{N+l}^{(\dagger)}}^{out}\st \frac{z}{\rep} \dt & \q=\emptyset
\end{cases}
\end{equation}

\item[] {\bf Transplanted potentials.} We need this term only when $\q=\emptyset$.   It is the  the term
\begin{equation}\label{eq:transplanted}
\textbf{P}_{\bg,\etag}:=\sum_{j=1}^{N}B_{j}^{2}\varepsilon^{2}\chi_{j}\psi_{\eta_{j}}\st \frac{z}{B_{j}	\varepsilon} \dt\,.
\end{equation}
with   the  coefficients $B_{j}$'s  defined in formula \eqref{eq:Bj} and $\psi_{\eta_{j}}$'s defined in formula \eqref{eq: poteta}.

\item[] {\bf Correction term.} It is the term that ensures the constancy of the scalar curvature of the metric $\omega_{\ag,\bg,\cg,\hg,\kg}$ and it is the solution on $M_{\rep}$ of the equations in the unknown $f$  
\begin{align}\label{eq:basegrezza2}
\Lg \sq f\dq=&\st 2s_{\omega}- \varepsilon^{2m-2}\nu_{{\bf 0},\cg}-2s_{\omega_{{\bf 0},\bg,\cg,\hg,\kg}}\dt-\Lg\sq \Pbe \dq -\Lg \sq \hko \dq \\
&+\NN _{\omega}\st -\varepsilon^{2m-2}{\bf{G}}_{{\bf 0},\bg,\cg}+\Pbe+ \hko+ f\dt\,,
\end{align}
if $\q=\emptyset$ and 

\begin{align}\label{eq:basegrezza2q}
\Lg \sq f\dq=&\st 2s_{\omega} +\varepsilon^{2m-2}\nu_{\ag,\cg}-2s_{\omega_{\ag,{\bf 0},\cg,\hg,\kg}}\dt -\Lg \sq \hko \dq \\
&+\NN _{\omega}\st +\varepsilon^{2m-2}{\bf{G}}_{\ag,{\bf 0},\cg}+ \hko+ f\dt\,.
\end{align}
if $\q\neq\emptyset$.

It is a function $f_{\ag,\bg,\cg,\hg,\kg}^{out}\in \Cqddq$ if $m\geq 3$ and $f_{\ag,\bg,\cg,\hg,\kg}^{out}\in \Cqdddq$ if $m=2$,  where the spaces $\Cqddq$ and $\Cqdddq$ are defined in  formulas \eqref{deficiency1} and \eqref{deficiency2}. 

\end{itemize} 

\begin{notation}
{\em For the rest of the paper we will denote with $\mathsf{C}$ a positive constant, that can vary from line to line, depending only on $\omega$, $\eta_{j}$'s and $\theta_{N+l}$'s.  }
\end{notation}

 We can now state the main propositions for the base space.  The first deals with the case $\q\neq \emptyset$, and its proof is an easy adaptation of the proof  of Proposition 5.1 in \cite{ap2}
\begin{prop}
\label{crucialbaseq}
Let $(M,g,\omega )$ a Kcsc orbifold with isolated singularities and let $\p$ be the set of singular points with non trivial orbifold group that admit a scalar flat ALE resolution with $\ega=0$ and $\q$ a non empty set of points admitting scalar flat ALE resolutions with $\ega\neq 0$ .

\begin{itemize}
\item  Assume exist $\ag\in \st\RR^{+}\dt^{K}$  such that 
	\begin{displaymath}
\left\{\begin{array}{lcl}
\sum_{l=1}^{K}a_{l}\frac{\egal}{|\egal|}\varphi_{i}\st q_{l} \dt=0 && i=1,\ldots, d\\
&&\\
\st \Xi\st \ag \dt\dt_{\substack{1\leq i\leq d\\ 1\leq l\leq K}}&& \textrm{has full rank}
\end{array}\right.
\end{displaymath}
where $\st\Xi\st \ag \dt\dt_{\substack{1\leq i\leq d\\ 1\leq l\leq K}}$ is the matrix introduced in Section \ref{linearanalysis} formula \eqref{eq:nondeggen}. Let ${\bf{G}}_{\ag,\bg,\cg}$ be the multi-poles solution of $\Lg$ introduced in Section \ref{linearanalysis}  Remark \ref{eq:gabc}.

\item    Let $\delta\in (4-2m,5-2m)$. Given any $\st \hg,\kg \dt\in \dombd$, where $\dombd$ is the space defined in formula \eqref{eq:dombdq},  let $\hko$ be the function defined in formula \eqref{eq:hko}. 
\begin{equation}
\hko:=\sum_{j=1}^{N}\chi_{j}H_{h_{j},k_{j}^{(\dagger)}}^{out}\st \frac{z}{\rep} \dt+\sum_{l=1}^{K}\chi_{l}H_{h_{N+l},k_{N+l}^{(\dagger)}}^{out}\st \frac{z}{\rep} \dt\,.
\end{equation}
\end{itemize}
Then for every $\left|\cg\right|\leq \rep^{4}$ there is $\bar{\ag}\in\st \RR^{+} \dt^{K}$ such that 
\begin{equation}
\left| \ag-\bar{\ag}\right|\leq \mathsf{C}\rep^{2m}\varepsilon^{2-2m}
\end{equation}
and $f_{\ag,{\bf 0},\cg,\hg,\kg}^{out}\in  \Cqdda$ if $m\geq 3$ and $f_{\ag,{\bf 0},\cg,\hg,\kg}^{out}\in C_{\delta}^{4,\alpha}\st M_{\p,\q} \dt\oplus\mathcal{E}_{\p}\oplus\mathcal{E}_{\q}\oplus \mathcal{D}_{\q}\st \bar{\ag} \dt$ if $m=2$ such that 
\begin{equation}
\omega_{\ag, {\bf 0},\cg,\hg,\kg}=\omega+i\dd\st {\bf{G}}_{\bar{\ag},\bf{0},\cg}  + \hko +f_{\ag,{\bf 0},\cg,\hg,\kg}^{out}  \dt
\end{equation}
is a Kcsc metric on $M_{\rep}$ and the following estimates hold.
\begin{equation}
\begin{array}{lll}
\left\|f_{\ag,\bf{0},\cg,\hg,\kg}^{out}\right\|_{\Cqdda}&\leq \mathsf{C} \rep^{2m+1}&\textrm{ for } m\geq 3\\
\left\|f_{\ag,\bf{0},\cg,\hg,\kg}^{out}\right\|_{C_{\delta}^{4,\alpha}\st M_{\p,\q} \dt\oplus\mathcal{E}_{\p}\oplus\mathcal{E}_{\q}\oplus \mathcal{D}_{\q}\st \bar{\ag} \dt}&\leq \mathsf{C} \rep^{5}&\textrm{ for } m= 2
\end{array}\,.
\end{equation}
 Moreover $s_{\omega_{\ag, {\bf 0},\cg,\hg,\kg}}$, the scalar curvature of $\omega_{\ag,{\bf 0},\cg,\hg,\kg}$, is a small perturbation of $s_{\omega}$, the scalar curvature of the background metric $\omega$ 
\begin{equation}
\left| s_{\omega_{\ag,{\bf 0},\cg,\hg,\kg}}-s_{\omega}\right|\leq \mathsf{C}\varepsilon^{2m-2}\,.
\end{equation}
\end{prop}

The second one deals with the case $\q=\emptyset$. The following Proposition follows from the same argument of the proof of \cite[Proposition 5.4]{alm} observing that Ricci flatness does not enter in the proof.

\begin{prop}
\label{crucialbase}
Let $(M,g,\omega )$ a Kcsc orbifold with isolated singularities and let $\p$ be the set of singular points with non trivial orbifold group that admit a scalar flat ALE  resolution with $\ega=0$.

\begin{itemize}
\item  Assume exist $\bg\in \st\RR^{+}\dt^{N}$ and $\cg\in \RR^{N}$ such that 
	\begin{displaymath}
\left\{\begin{array}{lcl}
\sum_{j=1}^{N} \frac{\cgaj}{|\cgaj|} \st b_{j}\Delta_{\omega}\varphi_{i}\st p_{j} \dt+c_{j}\varphi_{i}\st p_{j} \dt\dt=0 && i=1,\ldots, d\\
&&\\
\st\Theta\st \bg,\cg \dt\dt_{\substack{1\leq i\leq d\\ 1\leq j\leq N}}&& \textrm{has full rank}
\end{array}\right.
\end{displaymath}
where $\st\Theta\st \bg,\cg \dt\dt_{\substack{1\leq i\leq d\\ 1\leq j\leq N}}$ is the matrix introduced in Section \ref{linearanalysis} formula \eqref{eq:nondeggen}. Let ${\bf{G}}_{\bf{0},\bg,\cg}$ be the multi-poles solution of $\Lg$ introduced in Section \ref{linearanalysis}  in formula \ref{eq:gabc}.

\item    Let $\delta\in (4-2m,5-2m)$. Given any $\st \hg,\kg \dt\in \dombd$, where $\dombd$ is the space defined in formula \eqref{eq:dombd},  let $\hko$ be the function defined in formula \eqref{eq:hko}. 
\begin{equation}
\hko:=\sum_{j=1}^{N}\chi_{j}H_{h_{j}^{(\dagger)},k_{j}^{(\dagger)}}^{out}\st \frac{z}{\rep} \dt\,.
\end{equation}
\item Let ${\bf P}_{\bg,\etag}$ be the transplanted potentials defined in formula \eqref{eq:transplanted}
\begin{equation}
{\bf P}_{\bg,\etag}:=\sum_{j=1}^{N}B_{j}^{2}\varepsilon^{2}\chi_{j}\psi_{\eta_{j}}\st \frac{z}{B_{j}	\varepsilon} \dt\,.
\end{equation}
\end{itemize}
Then there is $f_{{\bf 0},\bg,\cg,\hg,\kg}^{out}\in \Cqdd$ if $m\geq 3$ and $f_{{\bf 0},\bg,\cg,\hg,\kg}^{out}\in \Cqddd$ if $m=2$ such that 
\begin{equation}
\omega_{{\bf 0},\bg,\cg,\hg,\kg}=\omega+i\dd\st {\bf{G}}_{\bf{0},\bg,\cg}  +{\bf P}_{\bg,\etag}+ \hko +f_{{\bf 0},\bg,\cg,\hg,\kg}^{out}  \dt
\end{equation}
is a Kcsc metric on $M_{\rep}$ and the following estimates hold.
\begin{equation}
\begin{array}{lll}
\left\|f_{\bf{0},\bg,\cg,\hg,\kg}^{out}\right\|_{\Cqdd}&\leq \mathsf{C} \varepsilon^{2m+2}\rep^{2-2m-\delta}&\textrm{ for } m\geq 3\\
\left\|f_{\bf{0},\bg,\cg,\hg,\kg}^{out}\right\|_{\Cqddd}&\leq \mathsf{C} \varepsilon^{6}\rep^{-2-\delta}&\textrm{ for } m= 2
\end{array}\,.
\end{equation}
 Moreover $s_{\omega_{{\bf 0},\bg,\cg,\hg,\kg}}$, the scalar curvature of $\omega_{{\bf 0},\bg,\cg,\hg,\kg}$, is a small perturbation of $s_{\omega}$, the scalar curvature of the background metric $\omega$ 
\begin{equation}
\left| s_{\omega_{{\bf 0},\bg,\cg,\hg,\kg}}-s_{\omega}\right|\leq \mathsf{C}\varepsilon^{2m}\,.
\end{equation}
\end{prop}

\subsection{Extremal metrics on the truncated model spaces}\label{extanalisinonlinearemodello} 

As in the base case  we want to construct perturbations of $\eta$ on $X_{\Gamma,\frac{\Rep}{a}}$ of the form
\begin{equation}
\eta_{\tilde{\hg},\tilde{\kg}}:=\varepsilon^{2}a^{2}\eta+i\varepsilon^{2}\dd  F_{\tilde{\hg},\tilde{\kg}}^{inn}\qquad a\in \RR^{+}
\end{equation}
such that $\tilde{\eta}_{\tilde{h},\tilde{k}}$ is extremal. Again, by Proposition \ref{variazioneestremale} we need to find   $ \st F_{\tilde{\hg},\tilde{\kg}}^{inn}, c^{inn}\dt\in C^{\infty}\st X_{\Gamma,\frac{\Rep}{a}} \dt^{T}\times \RR$ that solve equation \eqref{eq:estremalecompatta} i.e. 
\begin{equation}\label{eq:estremaletroncatomodello1}
P_{\eta}^{*}P_{\eta}\sq F_{\tilde{\hg},\tilde{\kg}}^{inn}\dq 	 =-2\varepsilon^{4}a^{4} c^{inn}- 2\varepsilon^{4}a^{6}\left<\mu_{\eta},\tilde{X}_{s}+\tilde{X}^{out}\right>  -\varepsilon^{4}a^{4}J\st \tilde{X}_{s}+\tilde{X}^{out} \dt\sq F_{\tilde{\hg},\tilde{\kg}}^{inn}\dq+a^{2}\NN_{\eta}\st \frac{1}{a^{2}} F_{\tilde{\hg},\tilde{\kg}}^{inn} \dt
\end{equation}

\begin{remark}
We want to point out that equation \eqref{eq:estremalecompatta1} on $X_{\Gamma,\frac{\Rep}{a}}$ and consequently equation \eqref{eq:estremaletroncatomodello1} make sense only because of Proposition \ref{potenzialiHamiltoniani}. Indeed, to write equation  \eqref{eq:estremaletroncatomodello1}, we need the moment map $\mu_{\eta}$ that produces Hamiltonian potentials of holomorphic vector fields on $X_{\Gamma}$ and Proposition \ref{potenzialiHamiltoniani} precisely ensures the existence of such Hamiltonian potentials.   
\end{remark}

We look for  $F_{\tilde{\hg},\tilde{\kg}}^{inn}$ of the form
\begin{align}
F_{\tilde{\hg},\tilde{\kg}}^{inn}:=&{\bf H}_{\tilde{h},\tilde{k}}^{inn}  + f_{\tilde{\hg},\tilde{\kg}}^{inn}\\
\end{align}
with 
\begin{align}
{\bf H}_{\tilde{h},\tilde{k}}^{inn}:=&\chi H_{\tilde{h},\tilde{k}}^{inn}	\st\frac{a x}{\Rep} \dt \\
\end{align}

As usual, it is convenient to work on the complete model $X_{\Gamma}$ and to use the standard truncation/extension operators $\mathcal{E}_{\Rep}$, so  we want to find
 $ \st f_{\tilde{\hg},\tilde{\kg}}^{inn}, c^{inn}\dt\in C_{\delta}^{4,\alpha}\st X_{\Gamma} \dt^{T}\times \RR$ with $\delta\in (0,1)$ such that 
\begin{align}\label{eq:estremalemodello1}
P_{\eta}^{*}P_{\eta}\sq  f_{\tilde{\hg},\tilde{\kg}}^{inn}\dq 	 =& -\mathcal{E}_{\frac{\Rep}{a}}P_{\eta}^{*}P_{\eta}\sq {\bf H}_{\tilde{h},\tilde{k}}^{inn}\dq-2\mathcal{E}_{\frac{\Rep}{a}}\varepsilon^{4}a^{4} c^{inn}- 2\varepsilon^{4}a^{6}\mathcal{E}_{\frac{\Rep}{a}}\left<\mu_{\eta},X\right> \\
 &-\varepsilon^{4}a^{4}\mathcal{E}_{\frac{\Rep}{a}}J\st \tilde{X}_{s}+\tilde{X}^{out}\dt\sq {\bf H}_{\tilde{h},\tilde{k}}^{inn}  + f_{\tilde{\hg},\tilde{\kg}}^{inn}\dq+a^{2}\mathcal{E}_{\frac{\Rep}{a}}\NN_{\eta}\st \frac{1}{a^{2}} \st {\bf H}_{\tilde{h},\tilde{k}}^{inn}  + f_{\tilde{\hg},\tilde{\kg}}^{inn} \dt \dt\,.
\end{align}

Making use of Proposition \ref{isomorfismopesati} we can we rephrase equation   \eqref{eq:estremalemodello1} as a fixed point problem
\begin{align}
\st  f_{\tilde{\hg},\tilde{\kg}}^{inn},c^{inn}  \dt\,\,=&\,\,\mathcal{N}^{inn}\st f_{\tilde{\hg},\tilde{\kg}}^{inn},c^{inn},\tilde{\hg},\tilde{\kg} \dt
\end{align}

with 
\begin{align}
\mathcal{N}^{inn}\,:&\,C_{\delta}^{4,\alpha}\st X_{\Gamma} \dt^{T}\times \RR\times \mathcal{B}\st \kappa,\delta \dt\rightarrow C_{\delta}^{4,\alpha}\st X_{\Gamma} \dt^{T}\times \RR
\end{align}
nonlinear continuous operator. With computations analogous to those performed in \cite{aps} it is possible to obtain the following result.

\begin{prop}\label{famigliamodello}
Let $\st X_{\Gamma},h,\eta \dt$  be a scalar-flat ALE resolutions of $\CC^{m}/\Gamma$ with $T$-invariant metric $h$. Then for every $\st \varepsilon^{2}\tilde{h},\varepsilon^{2}\tilde{k} \dt\in \mathcal{B}\st \kappa,\delta \dt$  and  $a\in \RR^{+}$,there is $\st  f_{\tilde{\hg},\tilde{\kg}}^{inn},c^{inn} \dt\in C_{\delta}^{4,\alpha}\st X_{\Gamma} \dt^{T}\times \RR$
with
\begin{equation}
\left\|f_{\tilde{\hg},\tilde{\kg}}^{inn}\right\|_{C_{\delta}^{4,\alpha}\st X_{\Gamma} \dt}+|c^{inn}|\leq C\st \kappa \dt \Rep^{4-2m}
\end{equation}
such that 
\begin{equation}
\eta_{\hg,\kg}:=a^{2}\varepsilon^{2}\eta+i\varepsilon^{2}\dd\st  {\bf H}_{\tilde{h},\tilde{k}}^{inn}  + f_{\tilde{\hg},\tilde{\kg}}^{inn} \dt
\end{equation}
is an  extremal K\"ahler metric on $X_{\Gamma,\frac{\Rep}{a}}$  with extremal vector field $\tilde{X}_{s}+\tilde{X}^{out}$ which is the natural lift of the vector field $X_{s}+X^{out}$ defined in Proposition \ref{famigliabase}. 
\end{prop}

As for the base case, an important consequence of Proposition \ref{famigliamodello} is that the following  estimate holds 
\begin{equation}
\left\| \left. \varepsilon^{2}f_{\tilde{\hg},\tilde{\kg}}^{inn}\st \frac{\Rep}{a}\cdot \dt \right|_{X_{\Gamma}\setminus X_{\Gamma,\frac{\Rep}{2a}} } \right\|_{C^{4,\alpha}\st \overline{B_{1}}\setminus B_{\frac{1}{2}} \dt} \leq \mathsf{C}\rep^{4}
\end{equation} 
with $\mathsf{C}$ positive constant depending only on $g,\eta$. Again, this kind of estimate is necessary for the success of the data matching procedure.

\subsection{Kcsc metrics on the truncated model spaces}\label{analisinonlinearemodello}

\noindent We now want to perform on the model spaces $X_{\Gamma_{j}}$'s and $Y_{\Gamma_{N+l}}$ a similar analysis as in the previous Subsection.  The constructions of the families of metrics are essentially the same made in \cite{alm} except for some complications due to the fact we require the models to be only scalar flat with $\ega=0$ and not necessarily Ricci-flat. These technical complications show up when we construct the {\em transplanted potential} and the {\em extensions of pseudo boundary data} and  are due to the presence of coefficients $\mathtt{c}_{0},\mathtt{c}_{2},\mathtt{c}_{3},\mathtt{c}_{4},\mathtt{c}_{5}$ relative to particular asymptotics of the potentials at infinity of the families of metrics which in the Ricci- flat case could be taken $\mathtt{c}_{0}=\mathtt{c}_{2}=\mathtt{c}_{3}=\mathtt{c}_{5}=0$ and $\mathtt{c_{4}}=-\frac{4\st m-1 \dt^{2}s_{\omega}}{m\st m+1\dt}$ . These coefficients, as we will see in the last section will influence the {\em balancing condition}.  

\begin{notation}
{\em To keep notations as short as possible we drop the subscripts $j$ and $l$. } 
\end{notation}
Our starting point are scalar-flat ALE K\"ahler manifold $\st X_{\Gamma},\eta,h \dt$ and $\st Y_{\Gamma},\theta,k \dt$ 
where we want to find $F_{\tilde{b},\tilde{h},\tilde{k}}^{in}\in C^{4,\alpha}\st X_{\Gamma,\frac{\Rep}{\tilde{b}}} \dt$ respectively $F_{\tilde{a},\tilde{h},\tilde{k}}^{in}\in C^{4,\alpha}\st Y_{\Gamma,\frac{\Rep}{\tilde{a}}} \dt$ with $\tilde{a},\tilde{b}\in \RR^{+}$ such that
\begin{equation}\theta_{\tilde{a},\tilde{h},\tilde{k}}:=\tilde{a}^{2}\theta+i\dd F_{\tilde{a},\tilde{h},\tilde{k}}^{in}\end{equation}
and
\begin{equation}\eta_{\tilde{b},\tilde{h},\tilde{k}}:=\tilde{b}^{2}\eta+i\dd F_{\tilde{b},\tilde{h},\tilde{k}}^{in}\end{equation}
are metrics on $Y_{\Gamma,\frac{\Rep}{\tilde{a}}}$ and $X_{\Gamma,\frac{\Rep}{\tilde{b}}}$. Moreover
\begin{equation}
\s_{\tilde{a}^{2}\theta}\st F_{\tilde{a},\tilde{h},\tilde{k}}^{in} \dt=\s_{\eta}\st F_{\tilde{b},\tilde{h},\tilde{k}}^{in} \dt=\varepsilon^{2}\st s_{\omega}+	\frac{1}{2}s_{\ag,{\bf 0},\cg,\hg,\kg}\dt
\end{equation}
when $\q\neq \emptyset$ and
\begin{equation}
\s_{\tilde{b}^{2}\eta}\st F_{\tilde{a},\tilde{h},\tilde{k}}^{in} \dt=\s_{\tilde{b}^{2}\eta}\st F_{\tilde{b},\tilde{h},\tilde{k}}^{in} \dt=\varepsilon^{2}\st s_{\omega}+	\frac{1}{2}s_{{\bf 0},\bg,\cg,\hg,\kg}\dt
\end{equation} 
when $\q=\emptyset$  with $\s_{\cdot}$  the operator introduced in \eqref{eq:espsg}.

\noindent The parameters $\tilde{a},\tilde{b}$ with the ``manual tuning" of the {\em principal asymptotics} and $\tilde{h}, \tilde{k}$ with the Cauchy data matching procedure. 
The modifications of the metrics $\eta$ and $\theta$   will be made of  blocks and it will  take different shapes if $\q\neq \emptyset$ or if $\q=\emptyset$. Indeed, if $\q \neq \emptyset$ then on $Y_{\Gamma}$ 

 \begin{equation}
F_{\tilde{a},\tilde{h},\tilde{k}}^{in}:= \hkii +f_{\tilde{a},\tilde{h},\tilde{k}}^{in}
\end{equation}

and on $X_{\Gamma}$

 \begin{equation}
F_{\tilde{b},\tilde{h},\tilde{k}}^{in}:= \hkii +f_{\tilde{b},\tilde{h},\tilde{k}}^{in}\,,
\end{equation}

if instead $\q=\emptyset$ then 

 \begin{equation}
F_{\tilde{b},\tilde{h},\tilde{k}}^{in}:=\Pbg+ \hkii +\csfii\,.
\end{equation}

The  term  $\Pbg$ is the transplanted potential of $\omega$   and comes into play when there are only $X_{\Gamma}$'s, $\hkii$ is the biharmonic extension of {\em pseudo-boundary data} $f_{\tilde{a},\tilde{h},\tilde{k}}^{in}, f_{\tilde{b},\tilde{h},\tilde{k}}^{in}$ are the perturbations ensuring the constancy of the scalar curvature. 

 \begin{itemize}

\item[] {\bf Transplanted potential.} If $\q=\emptyset$, then  we introduce the term $\Pbg$ that is a suitable modification of the function $\psi_{\omega}$ . Following exactly the same strategy of  \cite{alm} we look for functions $W_{4},W_{5}$ solutions of 
\begin{equation}
\begin{array}{ll}
\mathbb{L}_{\eta}\sq\Psi_{4}+W_{4}\dq&=-2s_{\omega}\\
&\\
\mathbb{L}_{\eta}\sq\Psi_{5}+W_{5}\dq&=0\,.
\end{array}
\end{equation}

\begin{notation}
 {\em For the rest of the subsection $\chi$ will denote a smooth cutoff function identically $0$ on $X_{\Gamma,\frac{R_{0}}{3\tilde{b}}}$ and identically $1$ outside $X_{\Gamma,\frac{R_{0}}{2\tilde{b}}}$.}
\end{notation}

\noindent We set 
\begin{equation}
\begin{array}{ll}
u_{4}&:= \begin{cases}
\st \frac{\Phi_{2}}{\Lambda_{2}^{2}}+\frac{\Phi_{4}}{\Lambda_{4}^{2}} \dt\chi|x|^{4-2m}&\textrm{for }m\geq3\\
&\\
\st \frac{\Phi_{2}}{\Lambda_{2}^{2}}+\frac{\Phi_{4}}{\Lambda_{4}^{2}} \dt\chi\log\st|x|\dt&\textrm{for }m=2
\end{cases}\\
&\\
u_{5}&:= \st \frac{\Phi_{3}}{\Lambda_{3}^{3}}+\frac{\Phi_{5}}{\Lambda_{5}^{2}} \dt\chi|x|^{5-2m}
\end{array}
\end{equation}

\noindent for a suitable choice of $\Phi_{2},\Phi_{4},\Phi_{3},\Phi_{5}$ eigenfunctions relative to the eigenvalues $\Lambda_{2},\Lambda_{4},\Lambda_{3},\Lambda_{5}$ of $\Delta_{\Sp^{2m-1}}$. 
Setting also
\begin{align}
\mathtt{c}_{4} := &\frac{|\Gamma|}{\cga |\Sp^{2m-1}|}\int_{X_{\Gamma}}\st\mathbb{L}_{\eta}\sq \chi \Psi_{4}+u_{4}\dq +2s_{\omega} \dt \, d\mu_{\eta}\label{eq:int4}\\
\mathtt{c}_{5}:= &\frac{|\Gamma|}{\cga |\Sp^{2m-1}|}\int_{X_{\Gamma}}\mathbb{L}_{\eta}\sq \chi \Psi_{5}+u_{5}\dq\,d\mu_{\eta} \label{eq:int5}
\end{align}
 we can find $v_{4}\in C_{\delta}^{4,\alpha}\st X_{\Gamma} \dt$ with $\delta\in (2-2m,3-2m)$ and $v_{5}\in C_{\delta}^{4,\alpha}\st X_{\Gamma} \dt$ with $\delta\in (3-2m,4-2m)$ such that 
\begin{displaymath}
\begin{array}{lcl}
\mathbb{L}_{\eta}\sq \chi \Psi_{4}+u_{4}-  \frac{\cga \mathtt{c}_{4}}{8\st m-2 \dt\st m-1 \dt}\chi |x|^{4-2m}+v_{4}\dq&=&-2s_{\omega}\qquad \textrm{ for }m\geq3\\
&&\\
\mathbb{L}_{\eta}\sq \chi \Psi_{4}+u_{4}+ \frac{\cga \mathtt{c_{4}}}{4}  \chi \log\st|x|\dt+v_{4}\dq&=&-2s_{\omega}\qquad \textrm{ for }m=2
\end{array}
\end{displaymath}

\begin{remark}
Contrarily to \cite{alm} here we do not have any information on constants $\mathtt{c}_{4},\mathtt{c}_{5}$, we do not know even their sign. If $X_{\Gamma}$ is Ricci-flat, as in \cite{alm}, one can to compute explicitly the constants $\mathtt{c}_{4}, \mathtt{c}_{5}$ and  show that $\mathtt{c}_{5}=0$ and $\mathtt{c}_{4}$ depends linearly on the scalar curvature of $M$ and nonlinearly on the dimension.   
\end{remark}
Now we can write the explicit expression of $W_{4}$
\begin{equation}
W_{4}:=\begin{cases}
-\frac{  \cga \mathtt{c}_{4} \tilde{b}^{4}\varepsilon^{2} }{8 \st m-2 \dt  \st m-1 \dt }\chi |x|^{4-2m}+u_{4}+v_{4}&\qquad \textrm{ for }m\geq3\,,\\
&\\
\frac{  \cga \mathtt{c}_{4}\tilde{b}^{4}\varepsilon^{2} }{4 }\chi \log\st|x|\dt+u_{4}+v_{4} & \qquad \textrm{ for }m=2\,.
\end{cases}
\end{equation}

\noindent  Analogously to the case of   $\Psi_{4}$ the correction $W_{5}$ of $\Psi_{5}$ is then
\begin{equation}
W_{5}:=\begin{cases}
-\frac{  \cga \mathtt{c}_{5} \tilde{b}^{5}\varepsilon^{3} }{8 \st m-2 \dt  \st m-1 \dt }\chi |x|^{4-2m}+u_{5}+v_{5}&\qquad \textrm{ for }m\geq3\,,\\
&\\
\frac{  \cga \mathtt{c}_{5}\tilde{b}^{5}\varepsilon^{3} }{4 }\chi \log\st|x|\dt+u_{5}+v_{5} & \qquad \textrm{ for }m=2\,.
\end{cases}
\end{equation}
 If we define  
\begin{equation}
V:=\varepsilon^{2}\tilde{b}^{4}W_{4}+\varepsilon^{3}\tilde{b}^{5}W_{5} \,.
\end{equation}
\noindent then we can define the transplanted potential $\Pbg$ as the function in $C^{4,\alpha}\st X_{\Gamma,\frac{\Rep}{\tilde{b}}} \dt$
\begin{equation}
\Pbg:=\begin{cases}
\frac{1}{\varepsilon^{2}}\chi\psi_{\omega}\st \tilde{b}\varepsilon x\dt +V&\qquad \textrm{ for }m\geq3\,,\\
&\\
\frac{1}{\varepsilon^{2}}\chi\psi_{\omega}\st \tilde{b}\varepsilon x \dt+V+C& \qquad \textrm{ for }m=2\,.
\end{cases}\label{eq:transplantedmod} 
\end{equation}
where $C$ is the constant term in the expansion  at $B_{2r_{0}}\st p \dt\setminus B_{\rep}\st p \dt$ of  
\begin{equation}
F_{{\bf 0},\bg,\cg,\hg,\kg}^{out}= - \varepsilon^{2m}{{\bf{G}}}_{\bf{0},\bg,\cg}+\Pbe+\hko+ f_{\bf{0},\bg,\cg,\hg,\kg}^{out}\,.
\end{equation}
introduced in Proposition \ref{crucialbase}.

\item[] {\bf Extensions of {\em pseudo-boundary data}.} Also this term takes different forms whether $\q=\emptyset$ or not.  If $\q\neq \emptyset$ then we define $\hkii\in C^{4,\alpha}\st Y_{\Gamma},\frac{\Rep}{\tilde{a}} \dt$ and $\hkii\in C^{4,\alpha}\st X_{\Gamma},\Rep \dt$ as
\begin{equation}\label{eq:thki2}
\hkii:=H_{\tilde{h},\tilde{k}}^{in}\st 0\dt+ \chi\st H_{\tilde{h},\tilde{k}}^{in}\st \frac{ \tilde{b} x}{\Rep} \dt-  H_{\tilde{h},\tilde{k}}^{in}\st 0\dt\dt\,.
\end{equation}
If instead $\q=\emptyset$ we need the construction of $\hkii$ performed in \cite{alm}.  Indeed, as for the transplanted potential, we look for functions $W_{0},W_{2},W_{3}$ for the equations
\begin{equation}
\begin{array}{ll}
\mathbb{L}_{\eta}\sq\chi|x|^{2}+W_{0}\dq&= 0\,, \\
&\\
\mathbb{L}_{\eta}\sq \chi|x|^{2}\Phi_{2}+W_{2}\dq&=0\,,\\
&\\
\mathbb{L}_{\eta}\sq\chi|x|^{3}\Phi_{3}+ W_{3}\dq&= 0\,.
\end{array}
\end{equation}
with $W_{0},W_{2},W_{3}$ having a structure similar to $W_{4},W_{5}$  we built for the transplanted potential. Indeed, once we set
\begin{align}
\mathtt{c}_{0}:=&\frac{|\Gamma|}{\cga |\Sp^{2m-1}|}\int_{X_{\Gamma}}\mathbb{L}_{\eta}\sq\chi|x|^{2}\dq d\mu_{\eta} \label{eq:c1}\\
\mathtt{c}_{2}:=&\frac{|\Gamma|}{\cga |\Sp^{2m-1}|}\int_{X_{\Gamma}}\mathbb{L}_{\eta}\sq \chi|x|^{2}\Phi_{2}\dq d\mu_{\eta}\label{eq:c2}\\
\mathtt{c}_{3}:=&\frac{|\Gamma|}{\cga |\Sp^{2m-1}|}\int_{X_{\Gamma}}\mathbb{L}_{\eta}\sq\chi|x|^{3}\Phi_{3}+ u^{(3)} \dq d\mu_{\eta}\label{eq:c3}\,
\end{align}
with \begin{equation}
u^{(3)}:=\chi|x|^{3-2m} \tilde{\Phi}_{3}
\end{equation}
\noindent for a suitable  spherical harmonic $\tilde{\Phi}_{3}$, we define 
\begin{equation}
W_{0}:=\begin{cases}
-\frac{\cga\mathtt{c}_{0}}{8\st m-2 \dt\st  m-1 \dt}\chi|x|^{4-2m}+v^{(0)}&m\geq 3\\
&\\
\frac{\cga\mathtt{c}_{0}}{4}\chi|x|^{4-2m}+v^{(0)}&m=2
\end{cases}
\end{equation}
\begin{equation}
W_{2}:=\begin{cases}
\frac{\cga\mathtt{c}_{2}}{8\st m-2 \dt\st  m-1 \dt}|x|^{4-2m}+v^{(2)}&m\geq 3\\
&\\
-\frac{\cga\mathtt{c}_{2}}{8}\chi|x|^{4-2m}+v^{(2)}&m= 2
\end{cases}
\end{equation}
\begin{equation}
W_{3}:=\begin{cases}
-\frac{\cga\mathtt{c}_{3}}{8\st m-2 \dt\st  m-1 \dt}|x|^{4-2m}+ u^{(3)} +v^{(3)}&m\geq 3\\
&\\
\frac{\cga\mathtt{c}_{3}}{4}|x|^{4-2m}+ u^{(3)} +v^{(3)}&m\geq 3\\
\end{cases}
\end{equation}

with $v^{(0)},v^{(2)},v^{(3)}\in C_{\delta}^{4,\alpha}\st X_{\Gamma} \dt$ for $\delta\in (2-2m,3-2m)$ . 

\begin{remark}
Here we see a technical complication due to the fact we ask $X_{\Gamma}$ to be only scalar flat with $\ega=0$ and not Ricci flat.
 Indeeed if $X_{\Gamma}$ were Ricci-flat, as in \cite{alm}, then the constants $\mathtt{c}_{0},\mathtt{c}_{2},\mathtt{c}_{3}$ would vanish, making the behaviour of the $W_{j}$'s easier.
\end{remark}

Moreover we set
\begin{equation}
V':=\frac{\tilde{k}^{(0)}\tilde{b}^2}{4m \Rep^2}W_{0} + \st \tilde{h}^{(2)}-\frac{\tilde{k}^{(2)}}{4(m+2)} \dt \frac{\tilde{b}^{2}}{\Rep^2}W_{2}+ \frac{\tilde{b}^{3}}{\Rep^3}\st \tilde{h}^{(3)}-\frac{\tilde{k}^{(3)}}{4(m+3)}  \dt  W_{3}
\end{equation}
and hence we can introduce the   function $\hkii \in C^{4,\alpha}(X_{\Gamma,\frac{\Rep}{\tilde{b}}})$ 
\begin{align}
\hkii :=& H_{\tilde{h},\tilde{k}}^{in}\st 0\dt+ \chi\st H_{\tilde{h},\tilde{k}}^{in}\st \frac{ \tilde{b} x}{\Rep} \dt-  H_{\tilde{h},\tilde{k}}^{in}\st 0\dt\dt+V'\,.\label{eq:thki}
\end{align}

\item[] {\bf Correction term.} The terms $f_{\tilde{a},\tilde{h},\tilde{k}}^{in}$ and $f_{\tilde{b},\tilde{h},\tilde{k}}^{in}$  that ensure the constancy of the scalar curvature of the metrics $\theta_{\tilde{a},\tilde{h},\tilde{k}}$ and of $\eta_{\tilde{b},\tilde{h},\tilde{k}}$ on $X_{\Gamma,\Rep}$  and are solutions of a  fixed point problem on a suitable closed and bounded subspace of $\Cdx$.
\end{itemize} 

\noindent We are now ready to state the main results on the model spaces.

\begin{prop}\label{crucialmodelloUm}
Let $\q\neq \emptyset$ and let $(X_{\Gamma},h,\eta )$ and $(Y_{\Gamma},k,\theta )$  ALE  scalar-flat K\"ahler resolutions of $\CC^{m}/\Gamma$ with $\Gamma$ finite subgroup of $U(m)$.
Let $\delta\in (0,1)$. Given any $\st \tilde{h},\tilde{k} \dt\in \mathcal{B}$, such that $\st \varepsilon^{2}\tilde{h},\varepsilon^{2}\tilde{k} \dt\in \dombd$, where $\dombd$ is the space defined in formula \eqref{eq:dombd},  let $\hkii$ be the function defined in formula \eqref{eq:thki2}. 
\begin{align}
\hkii :=& H_{\tilde{h},\tilde{k}}^{in}\st 0\dt+ \chi\st H_{\tilde{h},\tilde{k}}^{I}\st \frac{ \tilde{b} x}{\Rep} \dt-  H_{\tilde{h},\tilde{k}}^{I}\st 0\dt\dt\,.
\end{align}

Then there is $f_{\tilde{a},\tilde{h},\tilde{k}}^{in}\in C_{\delta}^{4,\alpha}\st Y_{\Gamma} \dt$ and $f_{\tilde{b},\tilde{h},\tilde{k}}^{in}\in C_{\delta}^{4,\alpha}\st X_{\Gamma} \dt$  such that 
\begin{equation}
\theta_{\tilde{a},\tilde{h},\tilde{k}}=\tilde{a}^{2}\theta+i\dd\st  \hkii + f_{\tilde{a},\tilde{h},\tilde{k}}^{in}  \dt
\end{equation}
is a Kcsc metric on $Y_{\Gamma,\frac{\Rep}{\tilde{a}}}$ and 
\begin{equation}
\eta_{\tilde{h},\tilde{k}}=\tilde{b}^{2}\eta+i\dd\st  \hkii + f_{\tilde{b},\tilde{h},\tilde{k}}^{in}  \dt
\end{equation}
is a Kcsc metric on $X_{\Gamma,\frac{\Rep}{\tilde{b}}}$. Moreover the following estimates hold.
\begin{equation}
\left\|\csfii\right\|_{\Cdx},\left\|f_{\tilde{a},\tilde{h},\tilde{k}}^{in}\right\|_{C_{\delta}^{4,\alpha}\st Y_{\Gamma }\dt}\leq \mathsf{C} \varepsilon^{2}\Rep^{4-\delta}
\end{equation}
 Moreover $s_{\eta_{\tilde{b},\tilde{h},\tilde{k}}} $ and $s_{\theta_{\tilde{a},\tilde{h},\tilde{k}}}$ satisfy
\begin{equation}
s_{\eta_{\tilde{h},\tilde{k}}}=s_{\theta_{\tilde{a},\tilde{h},\tilde{k}}} =s_{\omega_{\ag, {\bf 0},\cg,\hg,\kg}}\,.
\end{equation}
\end{prop}

The proof of Proposition \ref{crucialmodelloUm} is an esay adaptation of \cite[Lemma 5.3 ]{ap2}.

\begin{prop}\label{crucialmodello}
Let $\q=\emptyset$ and let $(X_{\Gamma},h,\eta )$ be a scalar flat ALE   K\"ahler resolution of $\CC^{m}/\Gamma$ with $\Gamma$ finite subgroup of $U(m)$ and $\ega=0$.

\begin{itemize}
\item    Let $\delta\in (4-2m,5-2m)$. Given any $\st \tilde{h},\tilde{k} \dt\in \mathcal{B}$, such that $\st \varepsilon^{2}\tilde{h},\varepsilon^{2}\tilde{k} \dt\in \dombd$, where $\dombd$ is the space defined in formula \eqref{eq:dombd},  let $\hkii$ be the function defined in formula \eqref{eq:thki}. 
\begin{align}
\hkii :=& H_{\tilde{h},\tilde{k}}^{in}\st 0\dt+ \chi\st H_{\tilde{h},\tilde{k}}^{in}\st \frac{ \tilde{b} x}{\Rep} \dt-  H_{\tilde{h},\tilde{k}}^{in}\st 0\dt\dt+V'\,.
\end{align}
\item Let $\Pbg$ be the transplanted potential defined in formula \eqref{eq:transplantedmod}
\begin{equation}
\Pbg:=\begin{cases}
\frac{1}{\varepsilon^{2}}\chi\psi_{\omega}\st \tilde{b}\varepsilon x\dt +V&\qquad \textrm{ for }m\geq3\,,\\
&\\
\frac{1}{\varepsilon^{2}}\chi\psi_{\omega}\st \tilde{b}\varepsilon x \dt+V+C& \qquad \textrm{ for }m=2\,.
\end{cases} 
\end{equation}
\end{itemize}
Then there is $\csfii\in \Cdx$  such that 
\begin{equation}
\etat=\tilde{b}^{2}\eta+i\dd\st \Pbg+ \hkii +\csfii  \dt
\end{equation}
is a Kcsc metric on $X_{\Gamma,\frac{\Rep}{\tilde{b}}}$ and the following estimate holds.
\begin{equation}
\left\|\csfii\right\|_{\Cdx}\leq C\st\kappa \dt \varepsilon^{2m+4}\rep^{-4m-\delta}\Rep^{-2}
\end{equation}
with $C\st \kappa \dt\in \RR^{+}$ depending only on $\omega$ and $\eta_{j}$'s and $\kappa$ the constant appearing in the definition of $\dombd$ (Section \ref{biharmonicext} formula \ref{eq:dombd} ). Moreover $s_{\etat}$, the scalar curvature of $\etat$ is 
\begin{equation}
s_{\etat}=s_{\omega_{{\bf 0},\bg,\hg,\kg}}\,.
\end{equation}
\end{prop}

The proof of Proposition \ref{crucialmodello} follows observing that the estimates in Lemmata 5.13, 5.14,5.15 in \cite{alm} hold also in this case precisely due to the choice of $\mathtt{c}_{0},\mathtt{c}_{2},\mathtt{c}_{3},\mathtt{c}_{5}$.

\section{Data matching}\label{matching}

\subsection{The extremal case}

\noindent We can now complete the proof of the following

\begin{teo}\label{maintheoremestremale}
Let $\st M,g,\omega\dt$ be a compact extremal orbifold with  $T$-invariant metric $g$. Let $\p$,$\q$ be as above. Then  there is $\bar{\varepsilon}$ such that for every $\varepsilon \in (0,\bar{\varepsilon})$ and $\bg\in \st\RR^{+}\dt^{N}$ and $\ag\in \st \RR^{+} \dt^{K}$ the orbifold
\[
\tilde{M} : = M \sqcup _{{p_{1}, \varepsilon}} X_{\Gamma_1} \sqcup_{{p_{2},\varepsilon}} \dots
\sqcup _{{p_N, \varepsilon}} X_{\Gamma_N}\sqcup _{{q_{1}, \varepsilon}} Y_{\Gamma_{N+1}} \sqcup_{{q_{2},\varepsilon}} \dots
\sqcup _{{q_K, \varepsilon}} Y_{\Gamma_{N+K}}
\]

\noindent has a $\tilde{T}$-invariant extremal K\"ahler metric in the class
\begin{equation} 
 \pi^*[\omega] + \sum_{l=1}^K\varepsilon^{2m-2} \tilde{a}_{l} ^{2m-2}[\tilde{\theta_{l}}] + \sum_{j=1}^N\varepsilon^{2m}b_{j} [\tilde{\eta_j}]
 \end{equation}   
where $\mathfrak{i}_{l}^{*}\sq \tilde{\theta}_{l} \dq=[\theta_{l}]$ with  $\mathfrak{i}_{l}:Y_{\Gamma_{N+l},\Rep}\hookrightarrow \tilde{M}$ the standard inclusion (and analogously for $\tilde{\eta}_{j}$).
\end{teo}

\begin{proof}[Proof of Theorem \ref{maintheoremestremale}]:  The result follows  combining Proposition \ref{famigliabase}, Proposition \ref{famigliamodello} and the standard procedure of data matching exposed in \cite[Section 10]{aps}.

\end{proof}

\subsection{The Kcsc case}

\noindent The final step is now to show a matching condition for  the metrics constructed above, and in particular why the quantities $\mathtt{c}_{0,j},\mathtt{c}_{2,j},\mathtt{c}_{3,j},\mathtt{c}_{5,j}$ do not interferee with the fixed point argument.

\begin{proof}[Proof of Theorem \ref{maintheorem}]: We focus on the case $m\geq 3$ since the proof for the case $m=2$ is exactly the same.  We denote  with $\mathcal{V}^{out}_{j,\ag,\bf{0},\cg,\hg,\kg}$ and  $\mathcal{V}^{out}_{l,\ag,\bf{0},\cg,\hg,\kg}$ the K\"ahler potentials of  $\omega_{\ag,\bf{0},\cg,\hg,\kg}$ on neighborhoods of points $p_{j}$ and $q_{l}$, with  $\mathcal{V}^{in}_{l,\tilde{a}_{l},\tilde{h}_{l},\tilde{k}_{l}}$ the K\"ahler potential of $\theta_{\tilde{a}_{l},\tilde{h}_{l},\tilde{k}_{l}}$ on $Y_{\Gamma_{N+l},\frac{\Rep}{\tilde{a_{l}}}}$ and with $\mathcal{V}^{in}_{j,\tilde{b}_{j},\tilde{h}_{j},\tilde{k}_{l}}$ the K\"ahler potential of $\eta_{\tilde{b}_{j},\tilde{h}_{j},\tilde{k}_{j}}$ on $X_{\Gamma_{j},\frac{\Rep}{\tilde{a_{j}}}}$.   It is possible, using  Propositions  \ref{crucialbaseq} and \ref{crucialmodelloUm}, to rescale, expand and decompose the potentials $\mathcal{V}^{out}_{j,\ag,\bf{0},\cg,\hg,\kg}$, $\mathcal{V}^{out}_{l,\ag,\bf{0},\cg,\hg,\kg}$, $\mathcal{V}^{in}_{l,\tilde{a}_{l},\tilde{h}_{l},\tilde{k}_{l}}$, $\mathcal{V}^{in}_{j,\tilde{b}_{j},\tilde{h}_{j},\tilde{k}_{l}}$ in the same way as done in \cite{alm}.

We start with the case   $\q\neq\emptyset$ and the``{\em tuning conditions}" are     
\begin{align}\label{eq:force1q}
 e\st \Gamma_{l} \dt  \tilde{a}_{l}^{2m-2}\varepsilon^{2}\Rep^{4-2m}=&  \st 1+\frac{\st f_{\ag,{\bf 0},\cg,\hg,\kg}^{out}\dt^{l}}{\varepsilon^{2m-2}} \dt e\st \Gamma_{l} \dt \bar{a}_{l}\varepsilon^{2m-2}\rep^{4-2m}-\frac{k_{l}^{(0)}}{4m-8}
\end{align}

\begin{align}\label{eq:force2q}
 C_{j}\varepsilon^{2m-2}\rep^{4-2m}=\frac{k_{j}^{(0)}}{4m-8}
\end{align}

where the quantities  $\st f_{\ag,{\bf 0},\cg,\hg,\kg}^{out}\dt^{l}$,  introduced in Section \ref{linearanalysis} , are the structural  coefficients  of $\sum_{l=1}^{K}\st f_{\ag,{\bf 0},\cg,\hg,\kg}^{out}\dt^{l}W_{\bar{\ag}}^{l}$
that is  the projection of $f_{\ag,{\bf 0},\cg,\hg,\kg}^{out}$ on the deficiency space $\mathcal{D}_{\q}\st\bar{\ag}\dt$. The ``{\em tuning conditions}" become hence
\begin{align}
  \tilde{a}_{l}^{2m-2}=&  \st 1+\frac{\st f_{\ag,{\bf 0},\cg,\hg,\kg}^{out}\dt^{l}}{\varepsilon^{2m-2}} \dt  \bar{a}_{l}-\frac{k_{l}^{(0)}\rep^{2m-4}}{\st 4m-8\dt e\st \Gamma_{l} \dt  \varepsilon^{2m-2}}
\end{align}
\begin{align}
 C_{j}=\frac{k_{j}^{(0)}\varepsilon^{2-2m}\rep^{2m-4}}{4m-8}
\end{align}

Now these conditions allow us to proceed exactly as in \cite[Subection 5.3]{ap2}  and  the proof of \ref{maintheorem} in the case $\q\neq \emptyset$ is complete. 

\medskip

\noindent Let now $\q=\emptyset$. Also in this case it is possible, using Propositions \ref{crucialbase}  and \ref{crucialmodello}, to rescale, expand and decompose the potentials $\mathcal{V}^{out}_{j,\bf{0},\bg,\cg,\hg,\kg}$, $\mathcal{V}^{in}_{j,\tilde{b}_{j},\tilde{h}_{j},\tilde{k}_{l}}$. The ``{\em tuning conditions}" we must impose are the following. 

\begin{align}\label{eq:merda1}
\cgaj  \tilde{b}_{j}^{2m}\varepsilon^{2}\Rep^{2-2m}=&\st 1-\frac{\st f_{{\bf 0},\bg,\cg,\hg,\kg}^{out}\dt^{j}}{\varepsilon^{2m}} \dt\cgaj  B_{j}^{2m}\varepsilon^{2m}\rep^{2-2m}+\st h_{j}^{(0)}+\frac{k_{j}^{(0)}}{4m-8}\dt \\
\end{align}

\begin{align}\label{eq:merda2}
- \st 1-\frac{\st f_{{\bf 0},\bg,\cg,\hg,\kg}^{out}\dt^{j}}{\varepsilon^{2m}} \dt C_{j}\varepsilon^{2m}\rep^{4-2m}-\frac{k_{j}^{(0)}}{4m-8}=&-\frac{  \cgaj \st\mathtt{c}_{4,j}+\tilde{b}_{j}\varepsilon\mathtt{c}_{5,j} \dt\varepsilon^{4}\tilde{b}^{2m}\Rep^{4-2m}}{8 \st m-2 \dt  \st m-1 \dt }\\
&-\frac{\cgaj\mathtt{c}_{0} \tilde{k}_{j}^{(0)}\tilde{b}_{j}^2}{32\st m-2 \dt\st m-1 \dt m \Rep^2}\\
&-\frac{\cgaj\mathtt{c}_{2}\tilde{b}_{j}^{2}}{\st m-2 \dt\st m-1 \dt\Rep^2}\st \tilde{h}_{j}^{(2)}-\frac{\tilde{k}_{j}^{(2)}}{4(m+2)} \dt \\
&-\frac{\cgaj\mathtt{c}_{3}\tilde{b}_{j}^{3}}{\st m-2 \dt\st m-1 \dt\Rep^3}\st \tilde{h}_{j}^{(3)}-\frac{\tilde{k}_{j}^{(3)}}{4(m+3)}  \dt\\
\end{align}

where the quantities  $\st f_{{\bf 0},\bg,\cg,\hg,\kg}^{out}\dt^{j}$, introduced in Section \ref{linearanalysis} , are the structural  coefficients  of $\sum_{j=1}^{N}\st f_{{\bf 0},\bg,\cg,\hg,\kg}^{out}\dt^{l}W_{\bg,\cg}^{j}$
that is  the projection of $f_{{\bf 0},\bg,\cg,\hg,\kg}^{out}$ on the deficiency space $\mathcal{D}_{\p}\st\bg,\cg\dt$. First we set 
\begin{equation}\label{eq:sceltaBj}
\tilde{b}_{j}^{2m}=B_{j}^{2m}\st 1-\frac{\st f_{{\bf 0},\bg,\cg,\hg,\kg}^{out}\dt^{j}}{\varepsilon^{2m}} \dt+\frac{1}{ \cgaj  }\st h_{j}^{(0)}+\frac{k_{j}^{(0)}}{4m-8}\dt\frac{\rep^{2m-2}}{\varepsilon^{2m}}
\end{equation}
 then we solve equation \eqref{eq:merda2} with respect to $C_{j}$ and  hence we determine the remaining tuning parameters. Letting $\varepsilon$ tend to $0$ we find, looking at the choices of tuning parameters,  the correct relation intertwining coefficients $\bg$ and $\cg$ in the {\em balancing condition} that is 
\begin{align}\label{eq:merda3}
 c_{j} =&  b_{j}\sq \frac{1}{m}s_{\omega}\st 1+\frac{\st m-1 \dt^{2}}{\st m+1 \dt} \dt -\frac{\mathtt{c}_{4,j}}{ 2  \st m-1 \dt|\Sp^{2m-1}|} \dq
\end{align}

Now, the proof of Theorem \ref{maintheorem} when $\q=\emptyset$ follows  arguing  as in   \cite[Subsection 6.2]{alm}. Indeed, after observing that the terms containing $\mathtt{c}_{0,j},\mathtt{c}_{2,j},\mathtt{c}_{3,j},\mathtt{c}_{5,j}$ have the $\varepsilon$-growth that is bigger than the $\varepsilon$-growth of {\em pseudo boundary data} and hence those terms  are dominated by the {\em pseudo boundary data}, one can use the argument of data matching exposed in \cite[ Section 6]{alm}.

\end{proof}

\providecommand{\bysame}{\leavevmode\hbox to3em{\hrulefill}\thinspace}
\providecommand{\MR}{\relax\ifhmode\unskip\space\fi MR }
\providecommand{\MRhref}[2]{%
  \href{http://www.ams.org/mathscinet-getitem?mr=#1}{#2}
}
\providecommand{\href}[2]{#2}

\end{document}